\documentclass[a4paper]{amsart}
\usepackage{color}
\usepackage[latin1]{inputenc}
\usepackage[T1]{fontenc}
\usepackage{amsfonts}
\usepackage{amssymb}
\usepackage{amsmath}
\usepackage{amsthm}
\usepackage{stmaryrd}
\usepackage{enumerate}
\usepackage{multirow}
\usepackage{graphicx}
\usepackage{setspace}
\usepackage{graphicx,type1cm,eso-pic,color}
\usepackage{pstricks}
\usepackage{float}
\newtheorem{theorem}{Theorem}[section]

\newtheorem{proposition}[theorem]{Proposition}

\newtheorem{remark}[theorem]{Remark}

\begin{document}
\setlength\arraycolsep{2pt}
\title[On a Computable Estimator of the Drift Parameter in Fractional SDE]{On a Computable Skorokhod's Integral Based Estimator of the Drift Parameter in Fractional SDE}
\author{Nicolas MARIE$^{\dag}$}
\address{$^{\dag}$Laboratoire Modal'X, Universit\'e Paris Nanterre, Nanterre, France}
\email{nmarie@parisnanterre.fr}
\keywords{Fractional Brownian motion; Least squares estimator; Malliavin calculus; Stochastic differential equations}
\date{}
\maketitle
\noindent
%


%
\begin{abstract}
This paper deals with a Skorokhod's integral based least squares type estimator $\widehat\theta_N$ of the drift parameter $\theta_0$ computed from $N\in\mathbb N^*$ (possibly dependent) copies $X^1,\dots,X^N$ of the solution $X$ of $dX_t =\theta_0b(X_t)dt +\sigma dB_t$, where $B$ is a fractional Brownian motion of Hurst index $H\in (1/3,1)$. On the one hand, some convergence results are established on $\widehat\theta_N$ when $H = 1/2$. On the other hand, when $H\neq 1/2$, Skorokhod's integral based estimators as $\widehat\theta_N$ cannot be computed from data, but in this paper some convergence results are established on a computable approximation of $\widehat\theta_N$.
\end{abstract}
\tableofcontents
%


%
\section{Introduction}\label{section_introduction}
Consider the stochastic differential equation
\begin{equation}\label{main_equation}
X_t = x_0 +\theta_0\int_{0}^{t}b(X_s)ds +\sigma B_t
\textrm{ $;$ }t\in [0,T],
\end{equation}
where $T > 0$ is fixed, $x_0\in\mathbb R$, $B = (B_t)_{t\in [0,T]}$ is a fractional Brownian motion of Hurst index $H\in (1/3,1)$, $b\in C^1(\mathbb R)$ (resp. $b\in C^2(\mathbb R)$) with bounded derivative(s) when $H\in [1/2,1)$ (resp. $H\in (1/3,1/2)$), $\sigma\in\mathbb R^*$ and $\theta_0\in\mathbb R$ is an unknown parameter to estimate. Under these conditions, Equation (\ref{main_equation}) has a unique (pathwise) solution $X = (X_t)_{t\in [0,T]}$. Throughout this paper, the parameters $H$ and $\sigma$ are assumed to be known.
\\
\\
The oldest kind of estimators of the drift function is based on the long-time behavior of the solution of Equation (\ref{main_equation}) or on small noise asymptotics. For $H = 1/2$, the reader may refer to the monograph \cite{KUTOYANTS04} written by Y. Kutoyants on long-time behavior based estimators, and to Kutoyants \cite{KUTOYANTS94} on small noise asymptotics based estimators. For $H\neq 1/2$, see Kubilius et al. \cite{KMR17}, Kleptsyna \& Le Breton \cite{KL01}, Tudor \& Viens \cite{TV07}, Hu \& Nualart \cite{HN10}, Neuenkirch \& Tindel \cite{NT14}, Hu et al. \cite{HNZ19}, Marie \& Raynaud de Fitte \cite{MRF21}, etc. on parametric long-time behavior based estimators, and see Saussereau \cite{SAUSSEREAU14} and Comte \& Marie \cite{CM19} on nonparametric ones. For small noise asymptotics based estimators when $H\neq 1/2$, the reader may refer to Mishra \& Prakasa Rao \cite{MPR11}, Marie \cite{MARIE20}, Nakajima \& Shimizu \cite{NS22}, etc. The stochastic integral involved in the definition of the estimators studied in \cite{HN10}, \cite{HNZ19}, \cite{MRF21} and \cite{CM19} is taken in the sense of Skorokhod. To be not computable from one observation of $X$ is the major drawback of the Skorokhod integral with respect to $X$. Precisely, the divergence operator from which the Skorokhod integral with respect to $B$ is defined is the adjoint of the Malliavin derivative. In general, this definition doesn't allow to compute the Skorokhod integral. The Skorokhod integral with respect to $B$ (or to $X$) can also be written as the limit of some Riemann's sums, where the usual product is replaced by the Wick product, which doesn't allow to compute it either (see Biagini et al. \cite{BHOZ08}, Subsection 2.2). One of the main purposes of our paper is to bypass this difficulty in the estimation framework presented below because the Skorokhod integral is a nice generalization of It\^o's integral, tailor-made for advanced statistical investigations.
\\
\\
For $H = 1/2$, a new kind of estimators of the drift function has been investigated since several years: those computed from $N$ copies $X^1,\dots,X^N$ of $X$ observed on $[0,T]$ with $T > 0$ fixed but $N\rightarrow\infty$. The major part of the literature deals with estimators based on independent copies of $X$ (see Comte \& Genon-Catalot \cite{CGC20}, Denis et al. \cite{DDM21}, Marie \& Rosier \cite{MR22}, etc.), but some recent papers are also devoted to estimators based on dependent copies (see Della Maestra \& Hoffmann \cite{DMH22} and Comte \& Marie \cite{CM22}). Copies based estimators are well-adapted to some situations difficult to manage with long-time behavior based estimators:
\begin{itemize}
 \item Assume that $X$ models the elimination process of a drug administered to one people, and assume that in a clinical-trial involving $N$ patients, $X^i$ models the elimination process of the same drug for the $i$-th patient. So, $X^1,\dots,X^N$ are independent copies of $X$, and then copies based estimators of $\theta_0$ are tailor-made in such situation. The reader may refer to Donnet \& Samson \cite{DS13}, Section 5 on such estimators in (population) pharmacokinetic/pharmacodynamic models. In fact, for any biological or physical process which can be recorded independently in $N$ individuals of same kind and modeled by $X$, independent copies based estimators of $\theta_0$ are appropriate.
 \item Consider a financial market with $N$ interacting risky assets of same kind and assume that the $i$-th asset is modeled by the $i$-th copy $X^i$ of the solution $X$ of Equation (\ref{main_equation}). Here, $X^1,\dots,X^N$ are not independent, but copies based estimators of the drift function remain appropriate (see Comte \& Marie \cite{CM22}). Such financial market model has been investigated in Duellmann et al. \cite{DKK10}. In this example, $H = 1/2$ in order to be able to avoid arbitrage opportunities.
\end{itemize}
For $H\neq 1/2$, Comte \& Marie \cite{CM21} and Marie \cite{MARIE23} are the only two references on such estimators up to our knowledge.
\\
\\
Our paper deals with the least squares type estimator
\begin{displaymath}
\widehat\theta_N :=
\left(\sum_{i = 1}^{N}\int_{0}^{T}b(X_{s}^{i})^2ds\right)^{-1}
\left(\sum_{i = 1}^{N}\int_{0}^{T}b(X_{s}^{i})\delta X_{s}^{i}\right)
\quad {\rm of}\quad\theta_0,
\end{displaymath}
where $N\in\mathbb N^*$, $X^i :=\mathcal I(x_0,B^i)$ for every $i\in\{1,\dots,N\}$, $B^1,\dots,B^N$ are copies of $B$, $\mathcal I(.)$ is the It\^o map for Equation (\ref{main_equation}), and the stochastic integral is taken in the sense of Skorokhod. Since $\widehat\theta_N$ is not computable when $H > 1/2$, our paper also deals with the estimator $\widetilde\theta_N$ approximating $\widehat\theta_N$ and defined as a fixed point:
\begin{eqnarray}
 \label{computable_estimator}
 \widetilde\theta_N & = &
 \frac{1}{NTD_N}\sum_{i = 1}^{N}
 \left[\int_{0}^{T}b(X_{s}^{i})
 dX_{s}^{i}\right.\\
 & &
 \hspace{3cm}
 \left. -
 \alpha_H\sigma^2\int_{0}^{T}\int_{0}^{t}
 b'(X_{t}^{i})\exp\left(\widetilde\theta_N
 \int_{s}^{t}b'(X_{u}^{i})du\right)|t - s|^{2H - 2}dsdt\right]
 \nonumber
\end{eqnarray}
where $\alpha_H := H(2H - 1)$,
\begin{displaymath}
D_N :=\frac{1}{NT}\sum_{i = 1}^{N}\int_{0}^{T}b(X_{s}^{i})^2ds
\end{displaymath}
and the stochastic integral in the right-hand side of Equation (\ref{computable_estimator}) is taken pathwise (in the sense of Young). The main purposes of our paper are:
\begin{enumerate}
 \item For $H = 1/2$, to provide an asymptotic confidence interval for $\widehat\theta_N$ when $X^1,\dots,X^N$ are independent, and then to establish risk bounds on truncated versions of $\widehat\theta_N$ and of a discrete-time approximation of $\widehat\theta_N$ when some copies of $X$ may depend on each other (see Section \ref{section_BM}). In Section \ref{section_BM}, $X$ is the solution of a stochastic differential equation driven by a multiplicative noise (in the sense of It\^o).
 \item For $H > 1/2$, to prove the existence and the uniqueness of the (fixed point) computable estimator $\widetilde\theta_N$ defined by (\ref{computable_estimator}), to establish the consistency and to provide an asymptotic confidence interval for $\widetilde\theta_N$ when $X^1,\dots,X^N$ are independent, and then to establish a risk bound on a truncated version of $\widetilde\theta_N$ when some copies of $X$ may depend on each other (see Section \ref{section_fBm}). In this case, $\widehat\theta_N$ is an auxiliary estimator.
 \item For $H\in (1/3,1/2)$, to extend the definition of the (fixed point) computable estimator $\widetilde\theta_N$ thanks to the relationship between the rough integral with respect to $B$ and the Skorokhod integral recently provided in Song \& Tindel \cite{ST22} (see Section \ref{section_fBm_rough}).
\end{enumerate}
Finally, Section \ref{section_numerical_experiments} deals with some numerical experiments on $\widetilde\theta_N$ when $H > 1/2$.
%


%
\section{Case $H = 1/2$}\label{section_BM}
Throughout this section, $H = 1/2$, and then the Skorokhod integral with respect to $B$ coincides with It\^o's integral on the space $\mathbb H^2$ of the adapted processes $U = (U_t)_{t\in [0,T]}$ such that
\begin{displaymath}
\int_{0}^{T}\mathbb E(U_{s}^{2})ds <\infty
\quad\textrm{(see Nualart \cite{NUALART06}, Proposition 1.3.11)}.
\end{displaymath}
In fact, in this section, there is no significant additional difficulties to consider a stochastic differential equation driven by a multiplicative noise. So, in the sequel, $(B^1,\dots,B^N)$ is a $N$-dimensional Brownian motion defined on a probability space $(\Omega,\mathcal F,\mathbb P)$, equipped with the filtration generated by $(B^1,\dots,B^N)$, and
\begin{equation}\label{SDE_BM}
X_{t}^{i} = x_0 +\theta_0\int_{0}^{t}b(X_{s}^{i})ds +\int_{0}^{t}\sigma(X_{s}^{i})\delta B_{s}^{i}
\textrm{ $;$ }t\in [0,T]\textrm{, }i\in\{1,\dots,N\},
\end{equation}
where $b,\sigma\in C^1(\mathbb R)$, $\sigma$, $b'$ and $\sigma'$ are bounded, $|\sigma(.)|\geqslant\mu$ with $\mu > 0$, and the stochastic integral is taken in the sense of It\^o. Under these conditions, Equation (\ref{SDE_BM}) has a unique (strong) solution $(X_{t}^{1},\dots,X_{t}^{N})_{t\in [0,T]}$. Moreover, consider
\begin{displaymath}
\mathcal R_N :=
\{(i,k)\in\{1,\dots,N\}^2 : i\neq k
\textrm{ and $X^i$ is not independent of $X^k$}\}.
\end{displaymath}
%


%
\begin{remark}\label{remark_R_N}
Let us make some remarks about $\mathcal R_N$:
\begin{enumerate}
 \item Since $|\sigma(.)| > 0$, and since $B^i$ is the only "source of randomness" in the definition of $X^i$ for every $i\in\{1,\dots,N\}$,
 \begin{displaymath}
 \mathcal R_N =
 \{(i,k)\in\{1,\dots,N\}^2 : i\neq k
 \textrm{ and $B^i$ is not independent of $B^k$}\}.
 \end{displaymath}
 \item Assume that there exists a correlation matrix $R$ such that, for every $i,k\in\{1,\dots,N\}$,
 \begin{displaymath}
 \mathbb E(B_{s}^{i}B_{t}^{k}) = R_{i,k}(s\wedge t)
 \textrm{ $;$ }\forall s,t\in [0,T].
 \end{displaymath}
 This leads to $d\langle B^i,B^k\rangle_t = R_{i,k}dt$ for every $i,k\in\{1,\dots,N\}$, and since $B^1,\dots,B^N$ are Gaussian processes,
 \begin{displaymath}
 \mathcal R_N =
 \{(i,k)\in\{1,\dots,N\}^2 : i\neq k\quad\textrm{and}\quad R_{i,k}\neq 0\}.
 \end{displaymath}
 Moreover, from continuous-time observations, $\sigma(.)$ may be assumed to be known, and then to determine the matrix $R$ is not a statistical problem. Indeed, since $|\sigma(.)| > 0$, for every $i,k\in\{1,\dots,N\}$,
 \begin{displaymath}
 R_{i,k} =\frac{\langle X^i,X^k\rangle_T}{\displaystyle{\int_{0}^{T}\sigma(X_{s}^{i})\sigma(X_{s}^{k})ds}}.
 \end{displaymath}
 \end{enumerate}
\end{remark}
\noindent
When $\mathcal R_N =\emptyset$, this section deals with an asymptotic confidence interval for $\widehat\theta_N$ (see Proposition \ref{ACI_BM}). When $\mathcal R_N\neq\emptyset$, a risk bound on a truncated version $\widehat\theta_{N}^{\mathfrak d}$ of $\widehat\theta_N$ is provided in Proposition \ref{risk_bound_BM}, and Proposition \ref{risk_bound_discrete_time_BM} deals with a risk bound on a discrete-time approximation of $\widehat\theta_{N}^{\mathfrak d}$.
\\
\\
Since $|\sigma(.)|\geqslant\mu$, and since $(b,\sigma)$ is Lipschitz continuous (because $b'$ and $\sigma'$ are bounded), for every $t\in (0,T]$, the common probability distribution of $X_{t}^{1},\dots,X_{t}^{N}$ has a density $f_t$ with respect to Lebesgue's measure such that, for every $x\in\mathbb R$,
\begin{equation}\label{Gaussian_bound_density_BM}
f_t(x)\leqslant
\mathfrak c_{0.5}t^{-\frac{1}{2}}
\exp\left[-\mathfrak m_{0.5}\frac{(x - x_0)^2}{t}\right]
\end{equation}
where $\mathfrak c_{0.5}$ and $\mathfrak m_{0.5}$ are positive constants depending on $T$ but not on $t$ and $x$ (see Menozzi et al. \cite{MPZ21}, Theorem 1.2). Then, $t\mapsto f_t(x)$ belongs to $\mathbb L^1([0,T])$, which legitimates to consider the density function $f$ defined by
\begin{displaymath}
f(x) :=\frac{1}{T}\int_{0}^{T}f_s(x)ds
\textrm{ $;$ }\forall x\in\mathbb R.
\end{displaymath}
Moreover, since $b'$ is bounded (and then $b$ has linear growth), still by Inequality (\ref{Gaussian_bound_density_BM}),
\begin{displaymath}
|b|^{\alpha} := |b(.)|^{\alpha}\in
\mathbb L^2(\mathbb R,f(x)dx)
\textrm{ $;$ }
\forall\alpha\in\mathbb R_+.
\end{displaymath}
{\bf Notation.} The usual norm on $\mathbb L^2(\mathbb R,f(x)dx)$ is denoted by $\|.\|_f$.
\\
\\
First, when $\mathcal R_N =\emptyset$, the following proposition provides an asymptotic confidence interval for $\widehat\theta_N$.
%


%
\begin{proposition}\label{ACI_BM}
If $\mathcal R_N =\emptyset$, then
\begin{displaymath}
\lim_{N\rightarrow\infty}
\mathbb P\left(\theta_0\in\left[
\widehat\theta_N -\frac{\overline Y_{N}^{1/2}}{\sqrt ND_N}u_{1 -\frac{\alpha}{2}}
\textrm{ $;$ }
\widehat\theta_N +\frac{\overline Y_{N}^{1/2}}{\sqrt ND_N}u_{1 -\frac{\alpha}{2}}\right]\right) =
1 -\alpha
\end{displaymath}
for every $\alpha\in (0,1)$, where $u_. :=\phi^{-1}(.)$, $\phi$ is the standard normal distribution function, and
\begin{displaymath}
\overline Y_N :=\frac{1}{NT^2}\sum_{i = 1}^{N}\int_{0}^{T}b(X_{s}^{i})^2\sigma(X_{s}^{i})^2ds.
\end{displaymath}
\end{proposition}
%


%
\begin{proof}
Consider
\begin{displaymath}
U_N :=\frac{1}{N}\sum_{i = 1}^{N}Z^i,
\end{displaymath}
where $Z^1,\dots,Z^N$ are the random variables defined by
\begin{displaymath}
Z^i :=\frac{1}{T}\int_{0}^{T}b(X_{s}^{i})\sigma(X_{s}^{i})\delta B_{s}^{i}
\textrm{ $;$ }\forall i\in\{1,\dots,N\}.
\end{displaymath}
First, since $\mathcal R_N =\emptyset$,
\begin{displaymath}
D_N =\frac{1}{NT}\sum_{i = 1}^{N}\int_{0}^{T}b(X_{s}^{i})^2ds
\xrightarrow[N\rightarrow\infty]{\mathbb P}
\mathbb E\left(\frac{1}{T}\int_{0}^{T}b(X_{s}^{1})^2ds\right) =\|b\|_{f}^{2}
\end{displaymath}
by the (usual) law of large numbers, and
\begin{displaymath}
\sqrt NU_N\xrightarrow[N\rightarrow\infty]{\mathcal D}\mathcal N(0,{\rm var}(Z^1))
\end{displaymath}
by the (usual) central limit theorem. Then, since $\delta X_{t}^{i} =\theta_0b(X_{t}^{i})dt +\sigma(X_{t}^{i})\delta B_{t}^{i}$ for every $i\in\{1,\dots,N\}$, and by Slutsky's lemma,
\begin{displaymath}
\sqrt N(\widehat\theta_N -\theta_0) =
\sqrt N\frac{U_N}{D_N}
\xrightarrow[N\rightarrow\infty]{\mathcal D}
\mathcal N\left(0,\frac{{\rm var}(Z^1)}{\|b\|_{f}^{4}}\right).
\end{displaymath}
Now, by the law of large numbers and the isometry property of It\^o's integral,
\begin{displaymath}
\overline Y_N
\xrightarrow[N\rightarrow\infty]{\mathbb P}
\frac{1}{T^2}\mathbb E\left(\int_{0}^{T}b(X_{s}^{1})^2\sigma(X_{s}^{1})^2ds\right) =
{\rm var}(Z^1).
\end{displaymath}
So,
\begin{displaymath}
\frac{\sqrt ND_N}{\overline Y_{N}^{1/2}}(\widehat\theta_N -\theta_0)
\xrightarrow[N\rightarrow\infty]{\mathcal D}\mathcal N(0,1)
\end{displaymath}
by Slutsky's lemma, and then
\begin{displaymath}
\lim_{N\rightarrow\infty}
\mathbb P\left(\frac{\sqrt ND_N}{\overline Y_{N}^{1/2}}|\widehat\theta_N -\theta_0|
\leqslant x\right) = 2\phi(x) - 1
\textrm{ $;$ }\forall x\in\mathbb R.
\end{displaymath}
\end{proof}
%


%
\begin{remark}\label{independent_copies_long_time_BM}
Assume that $\theta_0 > 0$ and that $b$ satisfies the following dissipativity condition:
\begin{equation}\label{dissipativity_condition}
\exists c > 0 :\forall x\in\mathbb R\textrm{, }b'(x)\leqslant -c.
\end{equation}
From one path of the solution $X$ of Equation (\ref{main_equation}) observed on $\mathbb R_+$, which seems to be a situation only appropriate for long-time behavior based estimators of $\theta_0$, one can construct $N$ independent copies of $X_{|[0,T]}$. To that purpose, consider the stopping times $\tau_1,\dots,\tau_N$ recursively defined by $\tau_1 = 0$ and
\begin{displaymath}
\tau_i =\inf\{t >\tau_{i - 1} + T : X_t = x_0\}
\textrm{ $;$ }i = 2,\dots,N
\end{displaymath}
with the convention $\inf(\emptyset) =\infty$. Since $\theta_0 > 0$ and $b$ fulfills (\ref{dissipativity_condition}), the scale density
\begin{displaymath}
s(.) :=\exp\left(-2\theta_0\int_{0}^{.}\frac{b(x)}{\sigma(x)^2}dx\right)
\end{displaymath}
satisfies
\begin{displaymath}
\int_{-\infty}^{0}s(x)dx =\int_{0}^{\infty}s(x)dx =\infty,
\end{displaymath}
and then $X$ is a recurrent Markov process by Khasminskii \cite{KHASMINSKII12}, Example 3.10. So, for any $i\in\{1,\dots,N\}$, $\mathbb P(\tau_i <\infty) = 1$ and one can consider the processes
\begin{displaymath}
B^i := (B_{\tau_i + t} - B_{\tau_i})_{t\in [0,T]}
\quad {\rm and}\quad
X^i := (X_{\tau_i + t})_{t\in [0,T]}.
\end{displaymath}
Since $B^1,\dots,B^N$ are independent Brownian motions by the strong Markov property, and since
\begin{displaymath}
X^i =\mathcal I(x_0,B^i)
\textrm{ $;$ }
\forall i\in\{1,\dots,N\},
\end{displaymath}
the processes $X^1,\dots,X^N$ are independent copies of $X_{|[0,T]}$.
\end{remark}
\noindent
Now, even when $\mathcal R_N\neq\emptyset$, the following proposition provides a suitable risk bound on the truncated estimator
\begin{displaymath}
\widehat\theta_{N}^{\mathfrak d} :=
\widehat\theta_N\mathbf 1_{D_N\geqslant\mathfrak d}
\quad {\rm with}\quad
\mathfrak d\in\Delta_f =\left(0,\frac{\|b\|_{f}^{2}}{2}\right].
\end{displaymath}
\noindent
The threshold $\mathfrak d$ is required in the proof of Proposition \ref{risk_bound_BM} in order to control the $\mathbb L^2$-risk of $\widehat\theta_{N}^{\mathfrak d}$ by the sum of those of $D_N$ and $D_N\widehat\theta_N$ up to a multiplicative constant. In practice, when possible, the threshold $\mathfrak d$ should be lower than the highest (easily) computable lower-bound on $\|b\|_{f}^{2}/2$ (see Remark \ref{remark_threshold_BM}).
%


%
\begin{proposition}\label{risk_bound_BM}
There exists a constant $\mathfrak c_{\ref{risk_bound_BM}} > 0$, not depending on $N$, such that
\begin{displaymath}
\mathbb E[(\widehat\theta_{N}^{\mathfrak d} -\theta_0)^2]
\leqslant
\frac{\mathfrak c_{\ref{risk_bound_BM}}}{N}
\left(1 +
\frac{|\mathcal R_N|}{N}\right).
\end{displaymath}
Precisely,
\begin{displaymath}
\mathfrak c_{\ref{risk_bound_BM}} =
\frac{1}{\mathfrak d^2}\left(
\frac{\|\sigma\|_{\infty}^{2}\|b\|_{f}^{2}}{T} +
\theta_{0}^{2}\|b^2\|_{f}^{2}\right).
\end{displaymath}
Moreover, if $|\mathcal R_N| = o(N^2)$, then both the truncated estimator $\widehat\theta_{N}^{\mathfrak d}$ and the main estimator $\widehat\theta_N$ are consistent.
\end{proposition}
%


%
\begin{proof}
The proof of Proposition \ref{risk_bound_BM} is dissected in two steps.
\\
\\
{\bf Step 1.} First of all, since $\delta X_{t}^{i} =\theta_0b(X_{t}^{i})dt +\sigma(X_{t}^{i})\delta B_{t}^{i}$ for every $i\in\{1,\dots,N\}$,
\begin{displaymath}
\widehat\theta_N =
\theta_0 +\frac{U_N}{D_N}
\quad {\rm with}\quad
U_N =
\frac{1}{NT}\sum_{i = 1}^{N}\int_{0}^{T}b(X_{s}^{i})\sigma(X_{s}^{i})\delta B_{s}^{i}.
\end{displaymath}
On the one hand, by Cauchy-Schwarz's inequality and the isometry property of It\^o's integral,
\begin{eqnarray}
 \mathbb E(U_{N}^{2}) & = &
 \frac{1}{N^2T^2}
 \sum_{i = 1}^{N}\mathbb E\left[
 \left(\int_{0}^{T}b(X_{s}^{i})\sigma(X_{s}^{i})\delta B_{s}^{i}\right)^2\right]
 \nonumber\\
 & &
 \hspace{2cm} +
 \frac{1}{N^2T^2}
 \sum_{(i,k)\in\mathcal R_N}\mathbb E\left(
 \left(\int_{0}^{T}b(X_{s}^{i})\sigma(X_{s}^{i})\delta B_{s}^{i}\right)
 \left(\int_{0}^{T}b(X_{s}^{k})\sigma(X_{s}^{k})\delta B_{s}^{k}\right)
 \right)
 \nonumber\\
 & \leqslant &
 \frac{1}{N^2T^2}\left(
 N\int_{0}^{T}\mathbb E(b(X_{s}^{1})^2\sigma(X_{s}^{1})^2)ds +
 |\mathcal R_N|\int_{0}^{T}\mathbb E(b(X_{s}^{1})^2\sigma(X_{s}^{1})^2)ds\right)
 \nonumber\\
 \label{risk_bound_BM_1}
 & \leqslant &
 \frac{\|\sigma\|_{\infty}^{2}}{NT}\left(
 1 +\frac{|\mathcal R_N|}{N}\right)
 \underbrace{\int_{-\infty}^{\infty}b(x)^2f(x)dx}_{=\|b\|_{f}^{2}}.
\end{eqnarray}
On the other hand,
\begin{displaymath}
\mathbb E(D_N) =
\frac{1}{T}\int_{0}^{T}\mathbb E(b(X_{s}^{1})^2)ds =\|b\|_{f}^{2},
\end{displaymath}
and then
\begin{eqnarray}
 \mathbb E[(D_N -\|b\|_{f}^{2})^2] =
 {\rm var}(D_N) & = &
 \frac{1}{N^2T^2}{\rm var}\left(
 \sum_{i = 1}^{N}\int_{0}^{T}b(X_{s}^{i})^2ds\right)
 \nonumber\\
 & = &
 \frac{1}{N}{\rm var}\left(\frac{1}{T}\int_{0}^{T}b(X_{s}^{1})^2ds\right)
 \nonumber\\
 & &
 \hspace{2cm} +
 \frac{1}{N^2}\sum_{(i,k)\in\mathcal R_N}{\rm cov}\left(
 \frac{1}{T}\int_{0}^{T}b(X_{s}^{i})^2ds,
 \frac{1}{T}\int_{0}^{T}b(X_{s}^{k})^2ds\right)
 \nonumber\\
 \label{risk_bound_BM_2}
 & \leqslant &
 \left(\frac{1}{N} +
 \frac{|\mathcal R_N|}{N^2}\right)
 \mathbb E\left[\left(\frac{1}{T}\int_{0}^{T}b(X_{s}^{1})^2ds\right)^2\right]
 \leqslant
 \frac{\|b^2\|_{f}^{2}}{N}
 \left(1 +
 \frac{|\mathcal R_N|}{N}\right).
\end{eqnarray}
Since
\begin{eqnarray*}
 \mathbb E[(\widehat\theta_{N}^{\mathfrak d} -\theta_0)^2] & = &
 \mathbb E[(\widehat\theta_N -\theta_0)^2\mathbf 1_{D_N\geqslant\mathfrak d}] +
 \theta_{0}^{2}\mathbb P(D_N <\mathfrak d)\\
 & \leqslant &
 \frac{1}{\mathfrak d^2}\mathbb E(U_{N}^{2}) +
 \theta_{0}^{2}\mathbb P\left(|D_N -\|b\|_{f}^{2}| >\frac{\|b\|_{f}^{2}}{2}\right)
 \leqslant
 \frac{1}{\mathfrak d^2}[\mathbb E(U_{N}^{2}) +
 \theta_{0}^{2}\mathbb E[(D_N -\|b\|_{f}^{2})^2]],
\end{eqnarray*}
by Inequalities (\ref{risk_bound_BM_1}) and (\ref{risk_bound_BM_2}),
\begin{displaymath}
\mathbb E[(\widehat\theta_{N}^{\mathfrak d} -\theta_0)^2]
\leqslant
\left(
\frac{\|\sigma\|_{\infty}^{2}\|b\|_{f}^{2}}{T} +
\theta_{0}^{2}\|b^2\|_{f}^{2}\right)
\frac{1}{\mathfrak d^2N}
\left(1 +
\frac{|\mathcal R_N|}{N}\right).
\end{displaymath}
{\bf Step 2.} Assume that $|\mathcal R_N| = o(N^2)$. First of all, the truncated estimator $\widehat\theta_{N}^{\mathfrak d}$ is consistent by Step 1. Now, since $\widehat\theta_N =\widehat\theta_{N}^{\mathfrak d}$ on $\{D_N\geqslant\mathfrak d\}$, for any $\varepsilon > 0$,
\begin{eqnarray*}
 \mathbb P(|\widehat\theta_N -\theta_0| >\varepsilon)
 & = &
 \mathbb P(\{|\widehat\theta_N -\theta_0| >\varepsilon\}
 \cap\{D_N <\mathfrak d\}) +
 \mathbb P(\{|\widehat\theta_{N}^{\mathfrak d} -\theta_0| >\varepsilon\}
 \cap\{D_N\geqslant\mathfrak d\})\\
 & \leqslant &
 \mathbb P(D_N <\mathfrak d) +
 \mathbb P(|\widehat\theta_{N}^{\mathfrak d} -\theta_0| >\varepsilon).
\end{eqnarray*}
On the one hand, since $\mathfrak d\in\Delta_f$, as already established in Step 1,
\begin{displaymath}
\mathbb P(D_N <\mathfrak d)\leqslant
\frac{\|b^2\|_{f}^{2}}{\mathfrak d^2N}
\left(1 +\frac{|\mathcal R_N|}{N}\right).
\end{displaymath}
On the other hand, by Markov's inequality and Step 1,
\begin{displaymath}
\mathbb P(|\widehat\theta_{N}^{\mathfrak d} -\theta_0| >\varepsilon)
\leqslant
\frac{1}{\varepsilon^2}\mathbb E[(\widehat\theta_{N}^{\mathfrak d} -\theta_0)^2]
\leqslant
\frac{\mathfrak c_{\ref{risk_bound_BM}}}{\varepsilon^2N}
\left(1 +
\frac{|\mathcal R_N|}{N}\right).
\end{displaymath}
Therefore,
\begin{displaymath}
\mathbb P(|\widehat\theta_N -\theta_0| >\varepsilon)
\leqslant
\frac{\mathfrak c_1}{N}\left(1 +
\frac{|\mathcal R_N|}{N}\right)
\quad {\rm with}\quad
\mathfrak c_1 =
\frac{\mathfrak c_{\ref{risk_bound_BM}}}{\varepsilon^2} +
\frac{\|b^2\|_{f}^{2}}{\mathfrak d^2}.
\end{displaymath}
In conclusion, since $|\mathcal R_N| = o(N^2)$,
\begin{displaymath}
\lim_{N\rightarrow\infty}
\mathbb P(|\widehat\theta_N -\theta_0| >\varepsilon) = 0.
\end{displaymath}
\end{proof}
%


%
\begin{remark}\label{remark_threshold_BM}
By (the second part of) Proposition \ref{risk_bound_BM}, $\widehat\theta_N$ and $\widehat\theta_{N}^{\mathfrak d}$ are both consistent, which means that for large values of $N$, one should always consider our main estimator $\widehat\theta_N$, especially when $\mathcal R_N =\emptyset$ because Proposition \ref{ACI_BM} provides (in addition) an asymptotic confidence interval for $\widehat\theta_N$. In this last situation, as usual, $N\geqslant 30$ may be considered as large. However, there are non-asymptotic theoretical guarantees on the truncated estimator $\widehat\theta_{N}^{\mathfrak d}$ thanks to the risk bound provided in (the first part of) Proposition \ref{risk_bound_BM}, but there are no such guarantees on $\widehat\theta_N$. So, in the non-asymptotic framework, let us make some remarks on how to choose the threshold $\mathfrak d$ in $\Delta_f = (0,\|b\|_{f}^{2}/2]$. When possible, the following rule should be observed: since the constant $\mathfrak c_{\ref{risk_bound_BM}}$ is of order $\mathfrak d^{-2}$, in order to degrade as few as possible the theoretical guarantees on $\widehat\theta_{N}^{\mathfrak d}$ provided by the risk bound in Proposition \ref{risk_bound_BM}, one should take $\mathfrak d =\mathfrak d_{\max}$, where $\mathfrak d_{\max}$ is the highest (easily) computable lower-bound on $\|b\|_{f}^{2}/2$. This doesn't mean that smaller values of $\mathfrak d$ cannot give better numerical results, but then theoretical guarantees of Proposition \ref{risk_bound_BM} are degraded. In some situations, let us show how to find a computable lower-bound $\mathfrak d_{\max}$ on $\|b\|_{f}^{2}/2$:
\begin{enumerate}
 \item Assume that $b(.)^2\geqslant\mathfrak b$ with a known $\mathfrak b > 0$. Since
 \begin{displaymath}
 D_N =\frac{1}{NT}\sum_{i = 1}^{N}
 \int_{0}^{T}b(X_{s}^{i})^2ds
 \geqslant\mathfrak b
 \quad {\rm and}\quad
 \|b\|_{f}^{2} =\int_{-\infty}^{\infty}b(x)^2f(x)dx
 \geqslant\mathfrak b,
 \end{displaymath}
 one should obviously take $\mathfrak d_{\max} =\mathfrak b/2$, and
 \begin{displaymath}
 \widehat\theta_{N}^{\mathfrak d} =\widehat\theta_N
 \quad {\rm for}\quad
 \mathfrak d =\mathfrak d_{\max}.
 \end{displaymath}
 So, in this special situation, there are also theoretical guarantees on our main estimator $\widehat\theta_N$ for small values of $N$.
 \item Assume that $X^1,\dots,X^N$ are Ornstein-Uhlenbeck processes (i.e. $b = -{\rm Id}_{\mathbb R}$), and let $m_t = x_0e^{-\theta_0t}$ (resp. $\sigma_t > 0$) be the common average (resp. standard deviation) of $X_{t}^{1},\dots,X_{t}^{N}$ for every $t\in(0,T]$. Let us provide a suitable lower bound on $\|b\|_{f}^{2}/2$ when $x_0\neq 0$. By assuming that there exists a known constant $\theta_{\max} > 0$ such that $\theta_0\in (0,\theta_{\max}]$,
 \begin{eqnarray*}
  \|b\|_{f}^{2} & = &
  \frac{1}{T}\int_{0}^{T}\int_{-\infty}^{\infty}b(x)^2f_s(x)dxds\\
  & = &
  \frac{1}{T}\int_{0}^{T}\frac{1}{\sigma_s\sqrt{2\pi}}
  \int_{-\infty}^{\infty}x^2\exp\left(-\frac{(x - m_s)^2}{2\sigma_{s}^{2}}\right)dxds\\
  & = &
  \frac{1}{T}\int_{0}^{T}\frac{1}{\sqrt{2\pi}}
  \int_{-\infty}^{\infty}(m_s +\sigma_sy)^2\exp\left(-\frac{y^2}{2}\right)dyds\\
  & = &
  \frac{1}{T}\int_{0}^{T}
  \frac{m_{s}^{2}}{\sqrt{2\pi}}\int_{-\infty}^{\infty}\exp\left(-\frac{y^2}{2}\right)dyds\\
  & &
  \hspace{2cm} +
  \frac{2}{T}\int_{0}^{T}
  \frac{m_s\sigma_s}{\sqrt{2\pi}}\int_{-\infty}^{\infty}y\exp\left(-\frac{y^2}{2}\right)dyds\\
  & &
  \hspace{4cm} +
  \frac{1}{T}\int_{0}^{T}
  \frac{\sigma_{s}^{2}}{\sqrt{2\pi}}\int_{-\infty}^{\infty}y^2\exp\left(-\frac{y^2}{2}\right)dyds\\
  & = &
  \frac{1}{T}\int_{0}^{T}m_{s}^{2}ds +
  \frac{1}{T}\int_{0}^{T}\sigma_{s}^{2}ds
  \geqslant x_{0}^{2}e^{-2\theta_{\max}T}.
 \end{eqnarray*}
 Then, one should take
 \begin{displaymath}
 \mathfrak d_{\max} =
 \frac{x_{0}^{2}}{2}e^{-2\theta_{\max}T}.
 \end{displaymath}
\end{enumerate}
\end{remark}
%


%
\begin{remark}\label{remark_risk_bound_BM}
If $|\mathcal R_N|\leqslant N$, then the $\mathbb L^2$-risk of $\widehat\theta_{N}^{\mathfrak d}$ remains of order $N^{-1/2}$ (parametric rate) as when $\mathcal R_N =\emptyset$. So, as in Comte and Marie \cite{CM22}, in the example of the financial market with $N$ interacting risky assets of same kind presented in the introduction section, our truncated estimator remains appropriate as long as the dependent assets are grouped in a "large enough" number of independent clusters. In the situation of Remark \ref{remark_R_N}.(2), this means that
 \begin{displaymath}
 R =
 \begin{pmatrix}
  R_1 & & (0)\\
   & \ddots & \\
  (0) & & R_{\frac{N}{q}}
 \end{pmatrix},
 \end{displaymath}
 where $q\in\mathbb N^*$, $N\in q\mathbb N^*$ and $R_1,\dots,R_{N/q}$ are $N/q$ correlation matrices of size $q\times q$.
\end{remark}
\noindent
Finally, for $n\in\mathbb N^*$, let us consider the following discrete-time approximations of $\widehat\theta_N$ and $\widehat\theta_{N}^{\mathfrak d}$ along the dissection $(t_0,\dots,t_n)$ of constant mesh $T/n$:
\begin{displaymath}
\widehat\theta_{N,n} :=\frac{V_{N,n}}{D_{N,n}}
\quad {\rm and}\quad
\widehat\theta_{N,n}^{\mathfrak d} :=
\widehat\theta_{N,n}
\mathbf 1_{D_{N,n}\geqslant\mathfrak d},
\end{displaymath}
where
\begin{eqnarray*}
 D_{N,n} & := &
 \frac{1}{NT}\sum_{i = 1}^{N}\sum_{j = 0}^{n - 1}
 b(X_{t_j}^{i})^2(t_{j + 1} - t_j)\\
 & &
 \hspace{2.5cm}{\rm and}\quad
 V_{N,n} :=
 \frac{1}{NT}\sum_{i = 1}^{N}\sum_{j = 0}^{n - 1}
 b(X_{t_j}^{i})(X_{t_{j + 1}} - X_{t_j}).
\end{eqnarray*}
The following proposition provides a suitable risk bound on $\widehat\theta_{N,n}^{\mathfrak d}$.
%


%
\begin{proposition}\label{risk_bound_discrete_time_BM}
If $b$ is bounded, then there exists a constant $\mathfrak c_{\ref{risk_bound_discrete_time_BM}} > 0$, not depending on $N$, $n$ and $\mathfrak d$, such that
\begin{displaymath}
\mathbb E[(\widehat\theta_{N,n}^{\mathfrak d} -\theta_0)^2]
\leqslant
\frac{\mathfrak c_{\ref{risk_bound_discrete_time_BM}}}{\mathfrak d^2}\left(
\frac{1}{N}\left(1 +\frac{|\mathcal R_N|}{N}\right) +\frac{1}{n}\right).
\end{displaymath}
Else, for any $\varepsilon > 0$, there exists a constant $\mathfrak c_{\ref{risk_bound_discrete_time_BM}}(\varepsilon) > 0$, not depending on $N$, $n$ and $\mathfrak d$, such that
\begin{displaymath}
\mathbb E[(\widehat\theta_{N,n}^{\mathfrak d} -\theta_0)^2]
\leqslant
\frac{\mathfrak c_{\ref{risk_bound_discrete_time_BM}}(\varepsilon)}{\mathfrak d^2}\left(
\frac{1}{N}\left(1 +\frac{|\mathcal R_N|}{N}\right) +\frac{1}{n^{1 -\varepsilon}}\right).
\end{displaymath}
Moreover, by assuming that $|\mathcal R_N| = o(N^2)$,
\begin{itemize}
 \item If $b$ is bounded, then both the truncated estimator
 \begin{displaymath}
 \widehat\theta_{N,\psi(N)}^{\mathfrak d}
 \quad {\rm with}\quad
 \psi(N) := N\left(1 +\frac{|\mathcal R_N|}{N}\right)^{-1},
 \end{displaymath}
 and the main estimator $\widehat\theta_{N,\psi(N)}$, are consistent.
 \item If $b$ may be unbounded, then both the truncated estimator
 \begin{displaymath}
 \widehat\theta_{N,\psi_{\varepsilon}(N)}^{\mathfrak d}
 \quad {\rm with}\quad
 \psi_{\varepsilon}(N) :=
 \left[N\left(1 +\frac{|\mathcal R_N|}{N}\right)^{-1}\right]^{\frac{1}{1 -\varepsilon}},
 \end{displaymath}
 and the main estimator $\widehat\theta_{N,\psi_{\varepsilon}(N)}$, are consistent.
\end{itemize}
\end{proposition}
%


%
\begin{proof}
The proof of Proposition \ref{risk_bound_discrete_time_BM} is dissected in three steps.
\\
\\
{\bf Step 1.} This step deals with suitable bounds on
\begin{displaymath}
\mathbb E[(V_N - V_{N,n})^2]
\quad {\rm with}\quad
V_N =\frac{1}{NT}\sum_{i = 1}^{N}\int_{0}^{T}b(X_{s}^{i})\delta X_{s}^{i}.
\end{displaymath}
First of all,
\begin{eqnarray*}
 |V_N - V_{N,n}| & = &
 \frac{1}{NT}\left|\sum_{i = 1}^{N}\sum_{j = 0}^{n - 1}
 \int_{t_j}^{t_{j + 1}}(b(X_{s}^{i}) - b(X_{t_j}^{i}))\delta X_{s}^{i}\right|\\
 & \leqslant &
 \frac{\theta_0}{NT}\left|\sum_{i = 1}^{N}\sum_{j = 0}^{n - 1}
 \int_{t_j}^{t_{j + 1}}(b(X_{s}^{i}) - b(X_{t_j}^{i}))b(X_{s}^{i})ds\right|\\
 & &
 \hspace{1cm} +
 \frac{1}{NT}\left|\sum_{i = 1}^{N}\sum_{j = 0}^{n - 1}
 \int_{t_j}^{t_{j + 1}}(b(X_{s}^{i}) - b(X_{t_j}^{i}))\sigma(X_{s}^{i})\delta B_{s}^{i}\right|
 =: B_{N,n} + A_{N,n}.
\end{eqnarray*}
On the one hand, let us control $\mathbb E(A_{N,n}^{2})$. By Cauchy-Schwarz's inequality and the isometry property of It\^o's integral,
\begin{eqnarray*}
 \mathbb E(A_{N,n}^{2}) & = &
 \frac{1}{NT^2}\mathbb E\left[\left(
 \int_{0}^{T}\left(\sum_{j = 0}^{n - 1}(b(X_{s}^{1}) - b(X_{t_j}^{1}))
 \sigma(X_{s}^{1})\mathbf 1_{(t_j,t_{j + 1})}(s)\right)\delta B_{s}^{1}
 \right)^2\right]\\
 & &
 \hspace{1cm}
 +\frac{1}{N^2T^2}\sum_{(i,k)\in\mathcal R_N}
 \mathbb E\left(\left(\int_{0}^{T}\left(\sum_{j = 0}^{n - 1}(b(X_{s}^{i}) - b(X_{t_j}^{i}))
 \sigma(X_{s}^{i})\mathbf 1_{(t_j,t_{j + 1})}(s)\right)\delta B_{s}^{i}\right)\right.\\
 & &
 \hspace{4cm}\left.\times
 \left(\int_{0}^{T}\left(\sum_{j = 0}^{n - 1}(b(X_{s}^{k}) - b(X_{t_j}^{k}))
 \sigma(X_{s}^{k})\mathbf 1_{(t_j,t_{j + 1})}(s)\right)\delta B_{s}^{k}\right)\right)\\
 & \leqslant &
 \frac{1}{NT^2}\left(1 +\frac{|\mathcal R_N|}{N}\right)
 \sum_{j = 0}^{n - 1}\int_{t_j}^{t_{j + 1}}
 \mathbb E[(b(X_{s}^{1}) - b(X_{t_j}^{1}))^2\sigma(X_{s}^{1})^2]ds.
\end{eqnarray*}
Moreover, for every $j\in\{0,\dots,n - 1\}$ and $s\in [t_j,t_{j + 1}]$, by the isometry property of It\^o's integral,
\begin{eqnarray*}
 \mathbb E[(b(X_{s}^{1}) - b(X_{t_j}^{1}))^2\sigma(X_{s}^{1})^2] & \leqslant &
 \|\sigma\|_{\infty}^{2}\|b'\|_{\infty}^{2}
 \mathbb E\left[\left(\theta_0\int_{t_j}^{s}b(X_{u}^{1})du +
 \int_{t_j}^{s}\sigma(X_{u}^{1})\delta B_{u}^{1}\right)^2\right]\\
 & \leqslant &
 2\|\sigma\|_{\infty}^{2}\|b'\|_{\infty}^{2}\left(\theta_{0}^{2}(s - t_j)\int_{t_j}^{s}
 \mathbb E(b(X_{u}^{1})^2)du\right.\\
 & &
 \hspace{4.5cm}\left. +
 \int_{t_j}^{s}\mathbb E(\sigma(X_{u}^{1})^2)du\right)\\
 & \leqslant &
 2\underbrace{\|\sigma\|_{\infty}^{2}\|b'\|_{\infty}^{2}
 (\theta_{0}^{2}T\|b\|_{f}^{2} +\|\sigma\|_{\infty}^{2})}_{=:\mathfrak c_1}(s - t_j).
\end{eqnarray*}
Then,
\begin{eqnarray*}
 \mathbb E(A_{N,n}^{2})
 & \leqslant &
 \frac{\mathfrak c_1}{NT^2}\left(1 +\frac{|\mathcal R_N|}{N}\right)
 \sum_{j = 0}^{n - 1}(t_{j + 1} - t_j)^2 =
 \frac{\mathfrak c_1}{N}\left(1 +\frac{|\mathcal R_N|}{N}\right)\frac{1}{n}.
\end{eqnarray*}
On the other hand, let us control $\mathbb E(B_{N,n}^{2})$ when $b$ may be unbounded, and then when $b$ is bounded:
\begin{itemize}
 \item For any $p,q > 0$ such that $1/p + 1/q = 1$, by H\"older's inequality,
 \begin{eqnarray*}
  \mathbb E(B_{N,n}^{2})
  & \leqslant &
  \frac{\theta_{0}^{2}}{T^2}\mathbb E\left[\left(\sum_{j = 0}^{n - 1}
  \int_{t_j}^{t_{j + 1}}(b(X_{s}^{1}) - b(X_{t_j}^{1}))b(X_{s}^{1})ds\right)^2\right]\\
  & \leqslant &
  \frac{\theta_{0}^{2}}{T}\sum_{j = 0}^{n - 1}
  \int_{t_j}^{t_{j + 1}}\mathbb E[(b(X_{s}^{1}) - b(X_{t_j}^{1}))^2b(X_{s}^{1})^2]ds\\
  & \leqslant &
  \frac{\theta_{0}^{2}\|b'\|_{\infty}^{2}}{T^{1 - 1/q}}\underbrace{\left(\frac{1}{T}\int_{0}^{T}
  \mathbb E(|b(X_{s}^{1})|^{2q})ds\right)^{\frac{1}{q}}}_{=\||b|^q\|_{f}^{2/q}}
  \sum_{j = 0}^{n - 1}\left(\int_{t_j}^{t_{j + 1}}
  \mathbb E(|X_{s}^{1} - X_{t_j}^{1}|^{2p})ds\right)^{\frac{1}{p}}.
 \end{eqnarray*}
 Moreover, by Burkholder-Davis-Gundy's inequality, there exist two positive constants $\mathfrak c_2$ and $\mathfrak c_3(p)$, not depending on $N$ and $n$, such that for every $j\in\{0,\dots,n - 1\}$ and $s\in [t_j,t_{j + 1}]$,
 \begin{eqnarray}
  \mathbb E(|X_{s}^{1} - X_{t_j}^{1}|^{2p})
  & \leqslant &
  \mathbb E\left[\left(\int_{t_j}^{s}|\theta_0b(X_{u}^{1})|du +
  \left|\int_{t_j}^{s}\sigma(X_{u}^{1})\delta B_{u}^{1}\right|\right)^{2p}\right]
  \nonumber\\
  & \leqslant &
  2^{2p - 1}|\theta_0|^{2p}(s - t_j)^{2p - 1}\int_{t_j}^{s}\mathbb E(|b(X_{u}^{1})|^{2p})du +
  2^{2p - 1}\mathbb E\left[\left|\int_{t_j}^{s}\sigma(X_{u}^{1})\delta B_{u}^{1}\right|^{2p}\right]
  \nonumber\\
  \label{risk_bound_discrete_time_BM_1}
  & \leqslant &
  2^{2p - 1}|\theta_0|^{2p}T\||b|^p\|_{f}^{2}(s - t_j)^{2p - 1} +
  2^{2p - 1}\mathfrak c_2\mathbb E\left[\left(\int_{t_j}^{s}\sigma(X_{u}^{1})^2du\right)^p\right]
  \leqslant
  \mathfrak c_3(p)(s - t_j)^p.
 \end{eqnarray}
 Then,
 \begin{eqnarray*}
  \mathbb E(B_{N,n}^{2})
  & \leqslant &
  \mathfrak c_3(p)^{\frac{1}{p}}\frac{\theta_{0}^{2}
  \|b'\|_{\infty}^{2}}{T^{1 - 1/q}}\||b|^q\|_{f}^{\frac{2}{q}}
  \sum_{j = 0}^{n - 1}\left(\int_{t_j}^{t_{j + 1}}(s - t_j)^pds\right)^{\frac{1}{p}}\\
  & = &
  \mathfrak c_3(p)^{\frac{1}{p}}
  \frac{\theta_{0}^{2}\|b'\|_{\infty}^{2}}{(p + 1)^{1/p}T^{1 - 1/q}}\||b|^q\|_{f}^{\frac{2}{q}}
  \sum_{j = 0}^{n - 1}(t_{j + 1} - t_j)^{1 +\frac{1}{p}} =
  \mathfrak c_3(p)^{\frac{1}{p}}
  \frac{\theta_{0}^{2}\|b'\|_{\infty}^{2}T}{(p + 1)^{1/p}}\||b|^q\|_{f}^{\frac{2}{q}}
  n^{-\frac{1}{p}}.
 \end{eqnarray*}
 So, for $p = 1/(1 -\varepsilon)$ and $q = 1/\varepsilon$, there exists a constant $\mathfrak c_4(\varepsilon) > 0$, not depending on $N$ and $n$, such that
 \begin{displaymath}
 \mathbb E(B_{N,n}^{2})\leqslant\frac{\mathfrak c_4(\varepsilon)}{n^{1 -\varepsilon}}.
 \end{displaymath}
 \item If $b$ is bounded, then
 \begin{eqnarray*}
  \mathbb E(B_{N,n}^{2})
  & \leqslant &
  \frac{\theta_{0}^{2}}{T}\sum_{j = 0}^{n - 1}
  \int_{t_j}^{t_{j + 1}}\mathbb E[(b(X_{s}^{1}) - b(X_{t_j}^{1}))^2b(X_{s}^{1})^2]ds\\
  & \leqslant &
  \frac{\theta_{0}^{2}\|b'\|_{\infty}^{2}\|b\|_{\infty}^{2}}{T}
  \sum_{j = 0}^{n - 1}\int_{t_j}^{t_{j + 1}}\mathbb E[(X_{s}^{1} - X_{t_j}^{1})^2]ds
 \end{eqnarray*}
 and, for every $j\in\{0,\dots,n - 1\}$ and $s\in [t_j,t_{j + 1}]$,
 \begin{eqnarray*}
  \mathbb E[(X_{s}^{1} - X_{t_j}^{1})^2]
  & \leqslant &
  2\theta_{0}^{2}\mathbb E\left[\left(\int_{t_j}^{s}b(X_{u}^{1})du\right)^2\right] +
  2\mathbb E\left(\int_{t_j}^{s}\sigma(X_{u}^{1})^2du\right)\\
  & \leqslant &
  2\underbrace{(\theta_{0}^{2}T\|b\|_{\infty}^{2} +
  \|\sigma\|_{\infty}^{2})}_{=:\mathfrak c_5}(s - t_j).
 \end{eqnarray*}
 So,
 \begin{displaymath}
 \mathbb E(B_{N,n}^{2})
 \leqslant
 \frac{\mathfrak c_5\theta_{0}^{2}\|b'\|_{\infty}^{2}\|b\|_{\infty}^{2}}{T}
 \sum_{j = 0}^{n - 1}(t_{j + 1} - t_j)^2 =
 \frac{\mathfrak c_6}{n}
 \quad {\rm with}\quad
 \mathfrak c_6 =\mathfrak c_5\theta_{0}^{2}\|b'\|_{\infty}^{2}\|b\|_{\infty}^{2}T.
 \end{displaymath}
\end{itemize}
Therefore,
\begin{displaymath}
\mathbb E[(V_{N,n} - V_N)^2]
\leqslant
2\left(\frac{\mathfrak c_1}{N}\left(1 +\frac{|\mathcal R_N|}{N}\right) +
\mathfrak c_4(\varepsilon)\right)\frac{1}{n^{1 -\varepsilon}}
\end{displaymath}
when $b$ may be unbounded, and
\begin{displaymath}
\mathbb E[(V_{N,n} - V_N)^2]
\leqslant
2\left(\frac{\mathfrak c_1}{N}\left(1 +\frac{|\mathcal R_N|}{N}\right) +
\mathfrak c_6\right)\frac{1}{n}
\end{displaymath}
when $b$ is bounded.
\\
\\
{\bf Step 2.} This step deals with suitable bounds on $\mathbb E[(D_N - D_{N,n})^2]$. First of all,
\begin{eqnarray*}
 |D_N - D_{N,n}| & = &
 \frac{1}{NT}\left|\sum_{i = 1}^{N}\sum_{j = 0}^{n - 1}
 \int_{t_j}^{t_{j + 1}}(b(X_{s}^{i})^2 - b(X_{t_j}^{i})^2)ds\right|\\
 & \leqslant &
 \frac{\|b'\|_{\infty}}{NT}\sum_{i = 1}^{N}\sum_{j = 0}^{n - 1}
 \int_{t_j}^{t_{j + 1}}(|b(X_{s}^{i})| + |b(X_{t_j}^{i})|)|X_{s}^{i} - X_{t_j}^{i}|ds.
\end{eqnarray*}
On the one hand, if $b$ is bounded, then
\begin{eqnarray*}
 \mathbb E[(D_N - D_{N,n})^2]
 & \leqslant &
 \frac{4\|b'\|_{\infty}^{2}\|b\|_{\infty}^{2}}{T}\sum_{j = 0}^{n - 1}
 \int_{t_j}^{t_{j + 1}}\mathbb E[(X_{s}^{1} - X_{t_j}^{1})^2]ds\\
 & \leqslant &
 \frac{4\mathfrak c_5\|b'\|_{\infty}^{2}\|b\|_{\infty}^{2}}{T}
 \sum_{j = 0}^{n - 1}(t_{j + 1} - t_j)^2 =
 \frac{4\mathfrak c_6}{n}.
\end{eqnarray*}
On the other hand, assume that $b$ may be unbounded. For every $q\in [1,\infty)$ and $s\in [0,T]$, by Inequality (\ref{risk_bound_discrete_time_BM_1}),
\begin{eqnarray*}
 \mathbb E(|b(X_{s}^{1})|^{2q})
 & \leqslant &
 2^{2q - 1}\|b'\|_{\infty}^{2q}\mathbb E(|X_{s}^{1} - x_0|^{2q}) + 2^{2q - 1}b(x_0)^{2q}\\
 & \leqslant &
 2^{2q - 1}(\|b'\|_{\infty}^{2q}\mathfrak c_3(q)T^q + b(x_0)^{2q}) =:\mathfrak c_7(q).
\end{eqnarray*}
Then, for any $p,q > 0$ such that $1/p + 1/q = 1$, by H\"older's inequality,
\begin{eqnarray*}
 \mathbb E[(D_N - D_{N,n})^2]
 & \leqslant &
 \frac{\|b'\|_{\infty}^{2}}{T}\sum_{j = 0}^{n - 1}
 \int_{t_j}^{t_{j + 1}}
 \mathbb E[(|b(X_{s}^{1})| + |b(X_{t_j}^{1})|)^{2q}]^{\frac{1}{q}}
 \mathbb E(|X_{s}^{1} - X_{t_j}^{1}|^{2p})^{\frac{1}{p}}ds\\
 & \leqslant &
 \frac{\|b'\|_{\infty}^{2}}{T}
 \cdot 4\mathfrak c_7(q)^{\frac{1}{q}}\cdot
 \sum_{j = 0}^{n - 1}\int_{t_j}^{t_{j + 1}}
 \mathbb E(|X_{s}^{1} - X_{t_j}^{1}|^{2p})^{\frac{1}{p}}ds.
\end{eqnarray*}
For $p = 1/(1 -\varepsilon)$ and $q = 1/\varepsilon$, by Inequality (\ref{risk_bound_discrete_time_BM_1}), there exists a constant $\mathfrak c_8(\varepsilon) > 0$, not depending on $N$ and $n$, such that
\begin{displaymath}
\mathbb E[(D_N - D_{N,n})^2]
\leqslant\frac{\mathfrak c_8(\varepsilon)}{n^{1 -\varepsilon}}.
\end{displaymath}
{\bf Step 3 (conclusion).} Assume that $b$ is bounded. By Proposition \ref{risk_bound_BM} together with the two previous steps, there exists a constant $\mathfrak c_9 > 0$, not depending on $N$ and $n$, such that
\begin{eqnarray*}
 \mathbb E[(V_{N,n} -\theta_0\|b\|_{f}^{2})^2]
 & \leqslant &
 2\mathbb E[(V_{N,n} - V_N)^2] +
 2\mathbb E[(U_N +\theta_0(D_N -\|b\|_{f}^{2}))^2]\\
 & \leqslant &
 \mathfrak c_9\left(
 \frac{1}{N}\left(1 +\frac{|\mathcal R_N|}{N}\right) +\frac{1}{n}\right)
\end{eqnarray*}
and
\begin{eqnarray*}
 \mathbb E[(D_{N,n} -\|b\|_{f}^{2})^2]
 & \leqslant &
 2\mathbb E[(D_{N,n} - D_N)^2] +
 2\mathbb E[(D_N -\|b\|_{f}^{2})^2]\\
 & \leqslant &
 \mathfrak c_9\left(
 \frac{1}{N}\left(1 +\frac{|\mathcal R_N|}{N}\right) +\frac{1}{n}\right).
\end{eqnarray*}
Therefore, there exists a constant $\mathfrak c_{10} > 0$, not depending on $N$ and $n$, such that
\begin{eqnarray*}
 \mathbb E[(\widehat\theta_{N,n}^{\mathfrak d} -\theta_0)^2] & = &
 \mathbb E\left[\left(\frac{V_{N,n}}{D_{N,n}}\mathbf 1_{D_{N,n}\geqslant\mathfrak d}
 -\theta_0\right)^2\right]\\
 & = &
 \mathbb E\left[\left(\frac{V_{N,n} -\theta_0D_{N,n}}{D_{N,n}}\right)^2
 \mathbf 1_{D_{N,n}\geqslant\mathfrak d}\right] +\theta_{0}^{2}
 \mathbb P(D_{N,n} <\mathfrak d)\\
 & \leqslant &
 \frac{1}{\mathfrak d^2}\mathbb E[(V_{N,n} -\theta_0\|b\|_{f}^{2}
 -\theta_0(D_{N,n} -\|b\|_{f}^{2}))^2]\\
 & &
 \hspace{4cm} +
 \theta_{0}^{2}\mathbb P\left(|D_{N,n} -\|b\|_{f}^{2}| >\frac{\|b\|_{f}^{2}}{2}\right)
 \quad\textrm{because $\mathfrak d\in\Delta_f$}\\
 & \leqslant &
 \frac{2}{\mathfrak d^2}[\mathbb E[(V_{N,n} -\theta_0\|b\|_{f}^{2})^2] +
 2\theta_{0}^{2}\mathbb E[(D_{N,n} -\|b\|_{f}^{2})^2]]
 \leqslant
 \frac{\mathfrak c_{10}}{\mathfrak d^2}\left(
 \frac{1}{N}\left(1 +\frac{|\mathcal R_N|}{N}\right) +\frac{1}{n}\right).
\end{eqnarray*}
The same way, if $b$ may be unbounded, then there exists a constant $\mathfrak c_{11}(\varepsilon) > 0$, not depending on $N$ and $n$, such that
\begin{displaymath}
\mathbb E[(\widehat\theta_{N,n}^{\mathfrak d} -\theta_0)^2]
\leqslant
\frac{\mathfrak c_{11}(\varepsilon)}{\mathfrak d^2}\left(
\frac{1}{N}\left(1 +\frac{|\mathcal R_N|}{N}\right) +\frac{1}{n^{1 -\varepsilon}}\right).
\end{displaymath}
Finally, assume that $|\mathcal R_N| = o(N^2)$. By following the same line than in the second step of the proof of Proposition \ref{risk_bound_BM}, thanks to the bounds on
\begin{displaymath}
\mathbb E[(D_{N,n} -\|b\|_{f}^{2})^2]
\quad {\rm and}\quad
\mathbb E[(\widehat\theta_{N,n}^{\mathfrak d} -\theta_0)^2]
\end{displaymath}
previously established, $\widehat\theta_{N,\psi(N)}^{\mathfrak d}$ and $\widehat\theta_{N,\psi(N)}$ (resp. $\widehat\theta_{N,\psi_{\varepsilon}(N)}^{\mathfrak d}$ and $\widehat\theta_{N,\psi_{\varepsilon}(N)}$) are consistent when $b$ is bounded (resp. $b$ may be unbounded).
\end{proof}
%


%
\section{Case $H > 1/2$}\label{section_fBm}
Subsection \ref{section_Skorokhod_integral} provides basics on Skorokhod's integral with respect to the fractional Brownian motion for $H\in (0,1)$, also useful in Section \ref{section_fBm_rough}, and then some more sophisticated results requiring $H\in (1/2,1)$. Subsection \ref{section_risk_bound_fBm} deals with a risk bound on a truncated version $\widehat\theta_{N}^{\mathfrak d}$ of the auxiliary estimator $\widehat\theta_N$, and Subsection \ref{section_computable_estimator} deals with the existence, the uniqueness and some convergence results on the computable estimator $\widetilde\theta_N$. Note that our estimation strategy in Subsection \ref{section_computable_estimator} is deeply related to the expression of the Malliavin derivative of $X_{t}^{1},\dots,X_{t}^{N}$, $t\in [0,T]$, when $\sigma$ is constant. This is why the case of a non-constant volatility function is not investigated when $H\neq 1/2$.
%


%
\subsection{Basics on the Skorokhod integral}\label{section_Skorokhod_integral}
Assume first that $H\in (0,1)$, and consider a $N$-dimensional fractional Brownian motion $(B^1,\dots,B^N)$ defined on a complete probability space $(\Omega,\mathcal F,\mathbb P)$, where $\mathcal F$ is the $\sigma$-algebra generated by $(B^1,\dots,B^N)$.
\\
\\
First, let us define the Skorokhod integral with respect to $B^i$ for any $i\in\{1,\dots,N\}$. Let $R$ be the covariance function of the fractional Brownian motion, which is defined by
\begin{displaymath}
R(s,t) :=\frac{1}{2}(s^{2H} + t^{2H} - |t - s|^{2H})
\textrm{ $;$ }\forall s,t\in [0,T],
\end{displaymath}
let $\mathcal H$ be the reproducing kernel Hilbert space of the fractional Brownian motion, that is the completion of the space of the step functions from $[0,T]$ into $\mathbb R$ with respect to the inner product $\langle .,.\rangle_{\mathcal H}$ defined by
\begin{displaymath}
\langle\mathbf 1_{[0,s]},
\mathbf 1_{[0,t]}\rangle_{\mathcal H} := R(s,t)
\textrm{ $;$ }\forall s,t\in [0,T],
\end{displaymath}
and let $\mathbf B^i$ be the (linear) isometry from $\mathcal H$ into $\mathbb L^2(\Omega)$ defined by
\begin{displaymath}
\mathbf B^i(h) :=
\int_{0}^{T}h(s)dB_{s}^{i}
\textrm{ $;$ }\forall h\in\mathcal H.
\end{displaymath}
Note that when $H\in (1/2,1)$, $\langle .,.\rangle_{\mathcal H}$ satisfies
\begin{displaymath}
\langle h,\eta\rangle_{\mathcal H} =
\alpha_H
\int_{0}^{T}\int_{0}^{T}
h(s)\eta(t)|t - s|^{2H - 2}dsdt
\textrm{ $;$ }\forall h,\eta\in\mathcal H.
\end{displaymath}
The Malliavin derivative of a smooth functional
\begin{displaymath}
F =\varphi(
\mathbf B^i(h_1),\dots,
\mathbf B^i(h_n))
\end{displaymath}
with $n\in\mathbb N^*$, $\varphi\in C_{\rm p}^{\infty}(\mathbb R^n;\mathbb R)$ (the space of all the smooth functions $\varphi :\mathbb R^n\rightarrow\mathbb R$ such that $\varphi$ and all its partial derivatives have polynomial growth) and $h_1,\dots,h_n\in\mathcal H$, is the $\mathcal H$-valued random variable
\begin{displaymath}
\mathbf D^iF :=
\sum_{j = 1}^{n}
\partial_j\varphi(
\mathbf B^i(h_1),\dots,
\mathbf B^i(h_n))h_j.
\end{displaymath}
By Nualart \cite{NUALART06}, Proposition 1.2.1, the map $\mathbf D^i$ is closable from $\mathbb L^2(\Omega;\mathbb R)$ into $\mathbb L^2(\Omega;\mathcal H)$. Its domain in $\mathbb L^2(\Omega;\mathbb R)$, denoted by $\mathbb D_{i}^{1,2}$, is the closure of the smooth functionals space for the norm $\|.\|_{1,2}$ defined by
\begin{displaymath}
\|F\|_{1,2}^{2} :=
\mathbb E(F^2) +
\mathbb E(\|\mathbf D^1F\|_{\mathcal H}^{2}).
\end{displaymath}
The adjoint $\delta^i$ of the Malliavin derivative $\mathbf D^i$ is called the divergence operator, its domain is denoted by ${\rm dom}(\delta^i)$, and for any process $Y = (Y_s)_{s\in [0,T]}$ and any $t\in (0,T]$ such that $Y\mathbf 1_{[0,t]}\in\textrm{dom}(\delta^i)$, the Skorokhod integral of $Y$ with respect to $B^i$ is defined on $[0,t]$ by
\begin{displaymath}
\int_{0}^{t}Y_s\delta B_{s}^{i} :=
\delta^i(Y\mathbf 1_{[0,t]}).
\end{displaymath}
Note that since $\delta^i$ is the adjoint of the Malliavin derivative $\mathbf D^i$, the Skorokhod integral of $Y$ with respect to $B^i$ on $[0,t]$ is a centered random variable:
\begin{equation}\label{zero_mean_Skorokhod_integral}
\mathbb E\left(\int_{0}^{t}Y_s\delta B_{s}^{i}\right) =
\mathbb E(1\cdot\delta^i(Y\mathbf 1_{[0,t]})) =
\mathbb E(\langle\mathbf D^i(1),Y\mathbf 1_{[0,t]}\rangle_{\mathcal H}) = 0.
\end{equation}
Let $\mathcal S^i$ be the space of the smooth functionals presented above and consider $\mathbb D_{i}^{1,2}(\mathcal H)$, the closure of
\begin{displaymath}
\mathcal S_{\mathcal H}^{i} :=
\left\{
\sum_{j = 1}^{n}F_jh_j
\textrm{ $;$ }
h_1,\dots,h_n\in\mathcal H
\textrm{, }
F_1,\dots,F_n\in\mathcal S^i
\right\}
\end{displaymath}
for the norm $\|.\|_{1,2,\mathcal H}$ defined by
\begin{displaymath}
\|Y\|_{1,2,\mathcal H}^{2} :=\mathbb E(\|Y\|_{\mathcal H}^{2}) +
\mathbb E(\|\mathbf D^1Y\|_{\mathcal H\otimes\mathcal H}^{2}).
\end{displaymath}
By Nualart \cite{NUALART06}, Proposition 1.3.1, $\mathbb D_{i}^{1,2}(\mathcal H)\subset {\rm dom}(\delta^i)$. Note that when $H\in (1/2,1)$, the divergence operator $\delta^i$ satisfies the following isometry type property on $\mathbb D_{i}^{1,2}(\mathcal H)$: for every $Y,Z\in\mathbb D_{i}^{1,2}(\mathcal H)$,
\begin{eqnarray}
 \label{isometry_divergence}
 \mathbb E(\delta^i(Y)\delta^i(Z)) & = &
 \alpha_H\int_{0}^{T}\int_{0}^{T}\mathbb E(Y_sZ_t)|t - s|^{2H - 2}dsdt\\
 & &
 \hspace{1.5cm} +
 \alpha_{H}^{2}\int_{[0,T]^4}\mathbb E(\mathbf D_{u'}^{i}Y_v
 \mathbf D_{v'}^{i}Z_u)|u - u'|^{2H - 2}|v - v'|^{2H - 2}du'dv'dudv.
 \nonumber
\end{eqnarray}
From now on and in the two following subsections, assume that $H\in (1/2,1)$, and consider
\begin{equation}\label{SDE_fBm}
X_{t}^{i} = x_0 +\theta_0\int_{0}^{t}b(X_{s}^{i})ds +\sigma B_{t}^{i}
\textrm{ $;$ }t\in [0,T]\textrm{, }i\in\{1,\dots,N\},
\end{equation}
where $b\in C^1(\mathbb R)$, $b'$ is bounded and $\sigma\neq 0$. Under these conditions, Equation (\ref{SDE_fBm}) has a unique (pathwise) solution $(X_{t}^{1},\dots,X_{t}^{N})_{t\in [0,T]}$. For any $i\in\{1,\dots,N\}$ and any $\alpha\in (0,H)$, the paths of $B^i$ are $\alpha$-H\"older continuous (see Nualart \cite{NUALART06}, Section 5.1), and so are those of $X^i$ by Equation (\ref{SDE_fBm}). Then, for every process $Y = (Y_t)_{t\in [0,T]}$ having $\beta$-H\"older continuous paths from $[0,T]$ into $\mathbb R$ with $\beta\in (0,1]$ such that $\alpha +\beta > 1$, the pathwise integral (in the sense of Young) of $Y$ with respect to $B^i$ (resp. $X^i$) on $[0,T]$ is well-defined (see Friz \& Victoir \cite{FV10}, Theorem 6.8):
\begin{eqnarray*}
 & &
 \int_{0}^{T}Y_sdB_{s}^{i} :=
\lim_{\pi(D)\rightarrow 0}\sum_{[u,v]\in D}Y_u(B_{v}^{i} - B_{u}^{i})\\
& &
\hspace{1cm}(\textrm{resp. }
\int_{0}^{T}Y_sdX_{s}^{i} :=
\lim_{\pi(D)\rightarrow 0}\sum_{[u,v]\in D}Y_u(X_{v}^{i} - X_{u}^{i}))
\textrm{ for any dissection $D$, of mesh $\pi(D)$, of $[0,T]$.}
\end{eqnarray*}
The two following propositions are crucial in order to establish a suitable risk bound on $\widehat\theta_N$ (see Subsection \ref{section_risk_bound_fBm}) and to compare $\widetilde\theta_N$ and $\widehat\theta_N$ (see Subsection \ref{section_computable_estimator}).
%


%
\begin{proposition}\label{Skorokhod_Young_relationship}
For every $\varphi\in C^1(\mathbb R)$ of bounded derivative, $(\varphi(X_{t}^{i}))_{t\in [0,T]}$ belongs to $\mathbb D_{i}^{1,2}(\mathcal H)$ and
\begin{eqnarray*}
 \int_{0}^{T}\varphi(X_{s}^{i})\delta X_{s}^{i}
 & := &
 \theta_0\int_{0}^{t}\varphi(X_{s}^{i})b(X_{s}^{i})ds +
 \sigma\int_{0}^{t}\varphi(X_{s}^{i})\delta B_{s}^{i}\\
 & = &
 \int_{0}^{T}\varphi(X_{s}^{i})dX_{s}^{i} -
 \alpha_H\sigma
 \int_{0}^{T}\int_{0}^{T}\mathbf D_{s}^{i}[\varphi(X_{t}^{i})]\cdot |t - s|^{2H - 2}dsdt.
\end{eqnarray*}
\end{proposition}
\noindent
Proposition \ref{Skorokhod_Young_relationship} is a straightforward consequence of Nualart \cite{NUALART06}, Proposition 5.2.3. Now, consider
\begin{displaymath}
M :=\sup_{x\in\mathbb R}b'(x).
\end{displaymath}
%


%
\begin{proposition}\label{bound_variance_Skorokhod}
There exists a constant $\mathfrak c_{\ref{bound_variance_Skorokhod}} > 0$, only depending on $H$ and $\sigma$, such that for every $\varphi\in C^1(\mathbb R)$ of bounded derivative,
\begin{eqnarray*}
 \mathbb E\left[\left(\int_{0}^{T}\varphi(X_{s}^{i})\delta B_{s}^{i}\right)^2\right]
 & \leqslant &
 \mathfrak c_{\ref{bound_variance_Skorokhod}}\overline{\mathfrak m}_{H,M,T}
 \left[\left(\int_{0}^{T}\mathbb E(|\varphi(X_{s}^{i})|^{\frac{1}{H}})ds\right)^{2H} +
 \left(\int_{0}^{T}\mathbb E(\varphi'(X_{s}^{i})^2)^{\frac{1}{2H}}ds\right)^{2H}
 \right]
\end{eqnarray*}
with $\overline{\mathfrak m}_{H,M,T} = 1\vee\mathfrak m_{H,M,T}$ and
\begin{displaymath}
\mathfrak m_{H,M,T} =
\left(-\frac{H}{M}\right)^{2H}\mathbf 1_{M < 0} +
T^{2H}\mathbf 1_{M = 0} +
\left(\frac{H}{M}\right)^{2H}e^{2MT}\mathbf 1_{M > 0}.
\end{displaymath}
\end{proposition}
\noindent
See Hu et al. \cite{HNZ19}, Proposition 4.4.(2) (proof) and Comte \& Marie \cite{CM21}, Theorem 2.9 for a proof.
%


%
\subsection{Risk bound on the auxiliary estimator $\widehat\theta_{N}^{\mathfrak d}$}\label{section_risk_bound_fBm}
As for $H = 1/2$, let us consider
\begin{eqnarray*}
 \mathcal R_N & := &
 \{(i,k)\in\{1,\dots,N\}^2 : i\neq k
 \textrm{ and $X^i$ is not independent of $X^k$}\}\\
 & = &
 \{(i,k)\in\{1,\dots,N\}^2 : i\neq k
 \textrm{ and $B^i$ is not independent of $B^k$}\}\\
 & = &
 \{(i,k)\in\{1,\dots,N\}^2 : i\neq k\textrm{ and }R_{i,k}\neq 0\}.
\end{eqnarray*}
Since $\sigma\neq 0$, and since $b'$ is bounded, for every $t\in (0,T]$, the common probability distribution of $X_{t}^{1},\dots,X_{t}^{N}$ has a density $f_t$ with respect to Lebesgue's measure such that, for every $x\in\mathbb R$,
\begin{equation}\label{Gaussian_bound_density_fBm}
f_t(x)\leqslant
\mathfrak c_Ht^{-H}
\exp\left[-\mathfrak m_H\frac{(x - x_0)^2}{t^{2H}}\right]
\end{equation}
where $\mathfrak c_H$ and $\mathfrak m_H$ are positive constants depending on $T$ but not on $t$ and $x$ (see Li et al. \cite{LPS23}, Theorem 1.3). This bound generalizes (\ref{Gaussian_bound_density_BM}) to fractional SDEs. Then, $t\mapsto f_t(x)$ belongs to $\mathbb L^1([0,T])$, which legitimates to consider the density function $f$ defined by
\begin{displaymath}
f(x) :=\frac{1}{T}\int_{0}^{T}f_s(x)ds\textrm{ $;$ }
\forall x\in\mathbb R.
\end{displaymath}
Moreover, since $b'$ is bounded (and then $b$ has linear growth), still by Inequality (\ref{Gaussian_bound_density_fBm}),
\begin{displaymath}
|b|^{\alpha}\in\mathbb L^2(\mathbb R,f(x)dx)
\textrm{ $;$ }
\forall\alpha\in\mathbb R_+.
\end{displaymath}
\noindent
As in Section \ref{section_BM}, the usual norm on $\mathbb L^2(\mathbb R,f(x)dx)$ is denoted by $\|.\|_f$. First, note that if $\mathcal R_N =\emptyset$, then by the (usual) law of large numbers and Equality (\ref{zero_mean_Skorokhod_integral}),
\begin{eqnarray}
 \label{convergence_fBm}
 \widehat\theta_N & = &
 \theta_0 +
 \left(\sum_{i = 1}^{N}\int_{0}^{T}b(X_{s}^{i})^2ds\right)^{-1}
 \left(\sum_{i = 1}^{N}\int_{0}^{T}b(X_{s}^{i})\delta B_{s}^{i}\right)\\
 & &
 \hspace{3cm}
 \xrightarrow[N\rightarrow\infty]{\mathbb P}
 \theta_0 +
 \frac{1}{\|b\|_{f}^{2}}
 \mathbb E\left(\frac{1}{T}\int_{0}^{T}b(X_{s}^{1})\delta B_{s}^{1}\right) =\theta_0.
 \nonumber
\end{eqnarray}
Now, even when $\mathcal R_N\neq\emptyset$, the following proposition provides a suitable risk bound on the truncated auxiliary estimator
\begin{displaymath}
\widehat\theta_{N}^{\mathfrak d} :=
\widehat\theta_N\mathbf 1_{D_N\geqslant\mathfrak d}
\quad {\rm with}\quad
\mathfrak d\in\Delta_f =\left(0,\frac{\|b\|_{f}^{2}}{2}\right].
\end{displaymath}
As when $H = 1/2$, the threshold $\mathfrak d$ is required in the proof of Proposition \ref{risk_bound_fBm} in order to control the $\mathbb L^2$-risk of $\widehat\theta_{N}^{\mathfrak d}$ by the sum of those of $D_N$ and $D_N\widehat\theta_N$ up to a multiplicative constant. In practice, when possible, the threshold $\mathfrak d$ should be lower than the highest (easily) computable lower-bound on $\|b\|_{f}^{2}/2$ (see Remark \ref{remark_threshold_fBm_part_1} for examples).
%


%
\begin{proposition}\label{risk_bound_fBm}
There exists a constant $\mathfrak c_{\ref{risk_bound_fBm}} > 0$, not depending on $N$, such that
\begin{displaymath}
\mathbb E[(\widehat\theta_{N}^{\mathfrak d} -\theta_0)^2]
\leqslant
\frac{\mathfrak c_{\ref{risk_bound_fBm}}}{N}
\left(1 +\frac{|\mathcal R_N|}{N}\right).
\end{displaymath}
Precisely,
\begin{displaymath}
\mathfrak c_{\ref{risk_bound_fBm}} =
\frac{1}{\mathfrak d^2}\left[
\frac{\sigma^2
\mathfrak c_{\ref{bound_variance_Skorokhod}}
\overline{\mathfrak m}_{H,M,T}}{T^{2 - 2H}}
\left[\left(\int_{-\infty}^{\infty}|b(x)|^{\frac{1}{H}}f(x)dx\right)^{2H} +
\int_{-\infty}^{\infty}b'(x)^2f(x)dx
\right] +\theta_{0}^{2}\|b^2\|_{f}^{2}\right].
\end{displaymath}
\end{proposition}
%


%
\begin{proof}
First of all, since $dX_{t}^{i} =\theta_0b(X_{t}^{i})dt +\sigma dB_{t}^{i}$ for every $i\in\{1,\dots,N\}$,
\begin{displaymath}
\widehat\theta_N =
\theta_0 +\frac{U_N}{D_N}
\quad {\rm with}\quad
U_N =
\frac{\sigma}{NT}\sum_{i = 1}^{N}\int_{0}^{T}b(X_{s}^{i})\delta B_{s}^{i}.
\end{displaymath}
By Cauchy-Schwarz's inequality, Proposition \ref{bound_variance_Skorokhod} and Jensen's inequality,
\begin{eqnarray*}
 \mathbb E(U_{N}^{2}) & = &
 \frac{\sigma^2}{N^2T^2}
 \sum_{i = 1}^{N}\mathbb E\left[
 \left(\int_{0}^{T}b(X_{s}^{i})\delta B_{s}^{i}\right)^2\right]\\
 & &
 \hspace{2cm} +
 \frac{\sigma^2}{N^2T^2}\sum_{(i,k)\in\mathcal R_N}
 \mathbb E\left(\left(\int_{0}^{T}b(X_{s}^{i})\delta B_{s}^{i}\right)
 \left(\int_{0}^{T}b(X_{s}^{k})\delta B_{s}^{k}\right)\right)\\
 & \leqslant &
 \frac{\sigma^2\mathfrak c_{\ref{bound_variance_Skorokhod}}
 \overline{\mathfrak m}_{H,M,T}}{NT^2}\left(1 +\frac{|\mathcal R_N|}{N}\right)\left[
 \left(\int_{0}^{T}\mathbb E(|b(X_{s}^{1})|^{\frac{1}{H}})ds\right)^{2H}\right.\\
 & &
 \left.\hspace{6cm} +
 \left(\int_{0}^{T}\mathbb E(b'(X_{s}^{1})^2)^{\frac{1}{2H}}ds\right)^{2H}
 \right]\\
 & \leqslant &
 \underbrace{
 \frac{\sigma^2\mathfrak c_{\ref{bound_variance_Skorokhod}}
 \overline{\mathfrak m}_{H,M,T}}{T^{2 - 2H}}
 \left[\left(\int_{-\infty}^{\infty}|b(x)|^{\frac{1}{H}}f(x)dx\right)^{2H} +
 \int_{-\infty}^{\infty}b'(x)^2f(x)dx
 \right]}_{=:\mathfrak c_1}\frac{1}{N}\left(1 +\frac{|\mathcal R_N|}{N}\right).
\end{eqnarray*}
As in the proof of Proposition \ref{risk_bound_BM},
\begin{displaymath}
\mathbb E[(D_N -\|b\|_{f}^{2})^2]
\leqslant
\frac{\|b^2\|_{f}^{2}}{N}\left(1 +\frac{|\mathcal R_N|}{N}\right),
\end{displaymath}
and then
\begin{eqnarray*}
 \mathbb E[(\widehat\theta_{N}^{\mathfrak d} -\theta_0)^2]
 & \leqslant &
 \frac{1}{\mathfrak d^2}[\mathbb E(U_{N}^{2}) +
 \theta_{0}^{2}\mathbb E[(D_N -\|b\|_{f}^{2})^2]]\\
 & \leqslant &
 \frac{\mathfrak c_2}{N}\left(1 +\frac{|\mathcal R_N|}{N}\right)
 \quad {\rm with}\quad
 \mathfrak c_2 =
 \frac{\mathfrak c_1 +\theta_{0}^{2}\|b^2\|_{f}^{2}}{\mathfrak d^2}.
\end{eqnarray*}
\end{proof}
\noindent
%


%
\begin{remark}\label{remark_threshold_fBm_part_1}
Let us conclude this section on how to find a computable lower-bound $\mathfrak d_{\max}$ on $\|b\|_{f}^{2}/2$ in some situations:
\begin{enumerate}
 \item Assume that $b(.)^2\geqslant\mathfrak b$ with a known $\mathfrak b > 0$. As for $H = 1/2$, since
 \begin{displaymath}
 D_N =\frac{1}{NT}\sum_{i = 1}^{N}
 \int_{0}^{T}b(X_{s}^{i})^2ds
 \geqslant\mathfrak b
 \quad {\rm and}\quad
 \|b\|_{f}^{2} =\int_{-\infty}^{\infty}b(x)^2f(x)dx
 \geqslant\mathfrak b,
 \end{displaymath}
 one should obviously take $\mathfrak d_{\max} =\mathfrak b/2$, and
 \begin{displaymath}
 \widehat\theta_{N}^{\mathfrak d} =\widehat\theta_N
 \quad {\rm for}\quad
 \mathfrak d =\mathfrak d_{\max}.
 \end{displaymath}
 \item Assume that $X^1,\dots,X^N$ are fractional Ornstein-Uhlenbeck processes (i.e. $b = -{\rm Id}_{\mathbb R}$), consider $T_{\max} > T$, and let $m_t = x_0e^{-\theta_0t}$ (resp. $\sigma_t > 0$) be the common average (resp. standard deviation) of $X_{t}^{1},\dots,X_{t}^{N}$ for every $t\in(0,T]$. Let us provide a suitable lower bound on $\|b\|_{f}^{2}/2$ when $x_0\neq 0$. By assuming that there exists a known constant $\theta_{\max} > 0$ such that $\theta_0\in (0,\theta_{\max}]$, as for $H = 1/2$,
 \begin{displaymath}
 \|b\|_{f}^{2} =
 \frac{1}{T}\int_{0}^{T}m_{s}^{2}ds +
 \frac{1}{T}\int_{0}^{T}\sigma_{s}^{2}ds
 \geqslant x_{0}^{2}e^{-2\theta_{\max}T}.
 \end{displaymath}
 Then, one should take
 \begin{displaymath}
 \mathfrak d_{\max} =
 \frac{x_{0}^{2}}{2}e^{-2\theta_{\max}T_{\max}}.
 \end{displaymath}
\end{enumerate}
\end{remark}
%


%
\subsection{A computable estimator}\label{section_computable_estimator}
When $H = 1/2$, $\widehat\theta_N$ is computable because as mentioned in Section \ref{risk_bound_BM}, the Skorokhod integral coincides with It\^o's integral on $\mathbb H^2$. When $H > 1/2$, the Skorokhod integral and then $\widehat\theta_N$ are not computable. However, this subsection deals with the approximation $\widetilde\theta_N$ of $\widehat\theta_N$ which is computable by solving Equation (\ref{computable_estimator}). First of all, let us explain why $\widetilde\theta_N$ is defined this way. For every $i\in\{1,\dots,N\}$, since
\begin{displaymath}
\mathbf D_{s}^{i}X_{t}^{i} =
\sigma\mathbf 1_{[0,t]}(s)\exp\left(\theta_0\int_{s}^{t}b'(X_{u}^{i})du\right)
\textrm{ $;$ }
\forall s,t\in [0,T],
\end{displaymath}
by Proposition \ref{Skorokhod_Young_relationship} and by the chain rule for the Malliavin derivative (see Nualart \cite{NUALART06}, Proposition 1.2.3),
\begin{eqnarray*}
 \int_{0}^{T}b(X_{s}^{i})\delta X_{s}^{i} & = &
 \int_{0}^{T}b(X_{s}^{i})dX_{s}^{i} -
 \alpha_H\sigma\int_{0}^{T}\int_{0}^{T}
 b'(X_{t}^{i})(\mathbf D_{s}^{i}X_{t}^{i})|t - s|^{2H - 2}dsdt\\
 & = &
 \int_{0}^{T}b(X_{s}^{i})dX_{s}^{i} -
 \alpha_H\sigma^2\int_{0}^{T}\int_{0}^{t}
 b'(X_{t}^{i})\exp\left(\theta_0\int_{s}^{t}b'(X_{u}^{i})du\right)|t - s|^{2H - 2}dsdt.
\end{eqnarray*}
Then, since $\widehat\theta_N$ is a converging estimator of $\theta_0$ as established in Subsection \ref{section_risk_bound_fBm},
\begin{displaymath}
\widehat\theta_N - I_N =
\Phi_N(\theta_0 - I_N)\approx
\Phi_N(\widehat\theta_N - I_N),
\end{displaymath}
where
\begin{displaymath}
\Phi_N(r) :=
-\frac{\alpha_H\sigma^2}{NTD_N}\sum_{i = 1}^{N}
\int_{0}^{T}\int_{0}^{t}
b'(X_{t}^{i})\exp\left[(r + I_N)
\int_{s}^{t}b'(X_{u}^{i})du\right]|t - s|^{2H - 2}dsdt
\end{displaymath}
and, by the change of variable formula for Young's integral,
\begin{eqnarray*}
 I_N & := &
 \frac{1}{NTD_N}
 \sum_{i = 1}^{N}\int_{0}^{T}b(X_{s}^{i})dX_{s}^{i}\\
 & = &
 \frac{1}{NTD_N}
 \sum_{i = 1}^{N}({\tt b}(X_{T}^{i}) - {\tt b}(x_0))
 \quad {\rm with}\quad
 {\tt b}'(.) = b(.).
\end{eqnarray*}
This legitimates to consider the estimator $\widetilde\theta_N := I_N + R_N$ of $\theta_0$, where $R_N$ is the fixed point of the map $\Phi_N$. Let us establish that $R_N$ exists and is unique under the condition (\ref{existence_uniqueness_R_N_1}) stated below.
%


%
\begin{proposition}\label{existence_uniqueness_R_N}
Assume that $b'(.)\leqslant 0$. If
\begin{equation}\label{existence_uniqueness_R_N_1}
T^{2H}\frac{M_N}{D_N}\leqslant
\frac{\mathfrak c}{\overline\alpha_H\sigma^2
\|b'\|_{\infty}^{2}},
\end{equation}
where $\mathfrak c$ is a deterministic constant arbitrarily chosen in $(0,1)$,
\begin{displaymath}
M_N :=
e^{\|b'\|_{\infty}|I_N|T}
\quad\textrm{and}\quad
\overline\alpha_H :=\frac{\alpha_H}{2H(2H + 1)},
\end{displaymath}
then $\Phi_N$ is a contraction from $\mathbb R_+$ into $\mathbb R_+$. Therefore, $R_N$ exists and is unique.
\end{proposition}
%


%
\begin{proof}
Since $b'(.)\leqslant 0$, $\Phi_N$ is nonnegative, and in particular $\Phi_N(\mathbb R_+)\subset\mathbb R_+$. Moreover, by (\ref{existence_uniqueness_R_N_1}), for every $r,\rho\in\mathbb R_+$,
\begin{eqnarray*}
 |\Phi_N(r) -\Phi_N(\rho)| & \leqslant &
 \frac{\alpha_H\sigma^2}{NTD_N}\sum_{i = 1}^{N}
 \int_{0}^{T}\int_{0}^{t}|t - s|^{2H - 2}|b'(X_{t}^{i})|
 \exp\left(I_N\int_{s}^{t}b'(X_{u}^{i})du\right)\\
 & &
 \hspace{3.5cm}\times
 \left|\exp\left(r\int_{s}^{t}b'(X_{u}^{i})du\right) -
 \exp\left(\rho
 \int_{s}^{t}b'(X_{u}^{i})du\right)\right|dsdt\\
 & \leqslant &
 \frac{\alpha_H\sigma^2}{NTD_N}
 \|b'\|_{\infty}M_N
 \sum_{i = 1}^{N}
 \int_{0}^{T}\int_{0}^{t}|t - s|^{2H - 2}\sup_{x\in\mathbb R_-}e^x\left|
 (r -\rho)\int_{s}^{t}b'(X_{u}^{i})du\right|dsdt\\
 & \leqslant &
 \overline\alpha_H\sigma^2\|b'\|_{\infty}^{2}
 T^{2H}\frac{M_N}{D_N}|r -\rho|
 \leqslant\mathfrak c|r -\rho|.
\end{eqnarray*}
So, $\Phi_N$ is a contraction from $\mathbb R_+$ into $\mathbb R_+$, and then $R_N$ exists and is unique by Picard's fixed point theorem.
\end{proof}
\noindent
When $\mathcal R_N =\emptyset$, the following proposition shows that the estimator $\widetilde\theta_{N}^{\mathfrak c} :=\widetilde\theta_N\mathbf 1_{\Omega_N}$, where
\begin{displaymath}
\Omega_N :=\left\{T^{2H}\frac{M_N}{D_N}\leqslant
\frac{\mathfrak c}{\overline\alpha_H\sigma^2
\|b'\|_{\infty}^{2}}\right\},
\end{displaymath}
is consistent.
%


%
\begin{proposition}\label{consistency_computable_estimator}
Assume that $b'(.)\leqslant 0$ and that $\theta_0 > 0$. If $\mathcal R_N =\emptyset$ and
\begin{equation}\label{consistency_computable_estimator_1}
\frac{T^{2H}}{\|b\|_{f}^{2}}
\exp\left(\frac{\|b'\|_{\infty}}{\|b\|_{f}^{2}}
|\mathbb E({\tt b}(X_{T}^{1})) - {\tt b}(x_0)|\right)
<\frac{\mathfrak c}{\overline\alpha_H\sigma^2
\|b'\|_{\infty}^{2}},
\end{equation}
then $\widetilde\theta_{N}^{\mathfrak c}\rightarrow\theta_0$ in probability when $N\rightarrow\infty$.
\end{proposition}
%


%
\begin{proof}
First, $\widehat\theta_N = I_N +\Phi_N(\theta_0 - I_N)$ and, on the event $\Omega_N$, $R_N =\widetilde\theta_N - I_N$ is the unique fixed point of the $\mathfrak c$-contraction $\Phi_N$ by Proposition \ref{existence_uniqueness_R_N}. Then,
\begin{eqnarray*}
 |\widetilde\theta_N -\widehat\theta_N|\mathbf 1_{\Omega_N} & = &
 |\Phi_N(R_N) -\Phi_N(\theta_0 - I_N)|\mathbf 1_{\Omega_N}\\
 & \leqslant &
 \mathfrak c|R_N - (\theta_0 - I_N)|\mathbf 1_{\Omega_N}
 \leqslant
 \mathfrak c|\widetilde\theta_N -\widehat\theta_N|\mathbf 1_{\Omega_N} +
 \mathfrak c|\widehat\theta_N -\theta_0|\mathbf 1_{\Omega_N}.
\end{eqnarray*}
Since $\mathfrak c\in (0,1)$,
\begin{displaymath}
|\widetilde\theta_N -\widehat\theta_N|\mathbf 1_{\Omega_N}
\leqslant
\frac{\mathfrak c}{1 -\mathfrak c}
|\widehat\theta_N -\theta_0|\mathbf 1_{\Omega_N},
\end{displaymath}
and thus
\begin{equation}\label{consistency_computable_estimator_2}
|\widetilde\theta_{N}^{\mathfrak c} -\theta_0| =
|\widetilde\theta_N -\theta_0|\mathbf 1_{\Omega_N} +
|\theta_0|\mathbf 1_{\Omega_{N}^{c}}
\leqslant
\frac{1}{1 -\mathfrak c}
|\widehat\theta_N -\theta_0| +
|\theta_0|\mathbf 1_{\Omega_{N}^{c}}.
\end{equation}
Now, by the (usual) law of large numbers,
\begin{displaymath}
D_N\xrightarrow[N\rightarrow\infty]{\mathbb P}
\frac{1}{T}\int_{0}^{T}\mathbb E(b(X_{s}^{1})^2)ds =\|b\|_{f}^{2} > 0
\end{displaymath}
and
\begin{displaymath}
\frac{1}{N}\sum_{i = 1}^{N}({\tt b}(X_{T}^{i}) - {\tt b}(x_0))
\xrightarrow[N\rightarrow\infty]{\mathbb P}
\mathbb E({\tt b}(X_{T}^{1})) - {\tt b}(x_0).
\end{displaymath}
Therefore,
\begin{eqnarray*}
 \frac{M_N}{D_N} & = &
 \frac{1}{D_N}\exp\left(\frac{\|b'\|_{\infty}}{ND_N}\left|
 \sum_{i = 1}^{N}({\tt b}(X_{T}^{i}) - {\tt b}(x_0))\right|\right)\\
 & &
 \hspace{3cm}
 \xrightarrow[N\rightarrow\infty]{\mathbb P}
 \frac{1}{\|b\|_{f}^{2}}
 \exp\left(\frac{\|b'\|_{\infty}}{\|b\|_{f}^{2}}
 |\mathbb E({\tt b}(X_{T}^{1})) - {\tt b}(x_0)|\right).
\end{eqnarray*}
When $N\rightarrow\infty$, this leads to $\mathbb P(\Omega_{N}^{c})\rightarrow 0$ by (\ref{consistency_computable_estimator_1}), and then to $\mathbf 1_{\Omega_{N}^{c}}\rightarrow 0$ in probability. In conclusion, by Inequality (\ref{consistency_computable_estimator_2}) together with the convergence result (\ref{convergence_fBm}),
\begin{displaymath}
|\widetilde\theta_{N}^{\mathfrak c} -\theta_0|
\xrightarrow[N\rightarrow\infty]{\mathbb P}0.
\end{displaymath}
\end{proof}
%


%
\begin{remark}\label{remark_condition_T}
The condition (\ref{consistency_computable_estimator_1}) in the statement of Proposition \ref{consistency_computable_estimator} may be thought of as a condition on the time horizon $T$ which can be chosen arbitrarily small in our estimation framework, even when $X^1,\dots,X^N$ have been observed on $[0,T_{\max}]$ with $0 < T\leqslant T_{\max}$. For instance, assume that there exists a (computable) constant $\ell > 0$, not depending on $T$, such that $\|b\|_f\geqslant\ell$ ($\ell =\sqrt{2\mathfrak d_{\max}}$ in Remark \ref{remark_threshold_fBm_part_1}), and let us show that there exists a (computable) constant $\mathfrak m_{\ell} > 0$, not depending on $T$, such that if $T <\mathfrak m_{\ell}$, then $T$ fulfills (\ref{consistency_computable_estimator_1}). By Taylor's formula with Lagrange's remainder, and since $b'$ is bounded,
\begin{eqnarray*}
 |\mathbb E({\tt b}(X_{T}^{1})) - {\tt b}(x_0)|
 & \leqslant &
 |b(x_0)|
 \mathbb E(|X_{T}^{1} - x_0|) +
 \frac{\|b'\|_{\infty}}{2}\mathbb E[(X_{T}^{1} - x_0)^2]\\
 & \leqslant &
 \mathfrak c_1[
 \mathbb E[(X_{T}^{1} - x_0)^2]^{\frac{1}{2}} +\mathbb E[(X_{T}^{1} - x_0)^2]]
 \quad {\rm with}\quad
 \mathfrak c_1 = |b(x_0)|\vee\frac{\|b'\|_{\infty}}{2}.
\end{eqnarray*}
Moreover,
\begin{eqnarray*}
 \mathbb E[(X_{T}^{1} - x_0)^2]
 & = &
 \mathbb E\left[\left(\theta_0\int_{0}^{T}b(X_{s}^{1})ds +\sigma B_{T}^{1}\right)^2\right]\\
 & \leqslant &
 2\theta_{0}^{2}T\int_{0}^{T}\mathbb E[b(X_{s}^{1})^2]ds +
 2\sigma^2\mathbb E(|B_{T}^{1}|^2) =
 2(\theta_{0}^{2}T^2\|b\|_{f}^{2} +\sigma^2T^{2H}).
\end{eqnarray*}
Then,
\begin{eqnarray*}
 \exp\left(\frac{\|b'\|_{\infty}}{\|b\|_{f}^{2}}
 |\mathbb E({\tt b}(X_{T}^{1})) - {\tt b}(x_0)|\right)
 & \leqslant &
 \exp\left(\frac{\mathfrak c_1\|b'\|_{\infty}}{\|b\|_{f}^{2}}
 (\sqrt 2(\theta_0T\|b\|_f +\sigma T^H) +
 2(\theta_{0}^{2}T^2\|b\|_{f}^{2} +\sigma^2T^{2H}))\right)\\
 & \leqslant &
 \exp\left(2\mathfrak c_1\|b'\|_{\infty}\left(
 \theta_{0}^{2}T^2 +
 \frac{\theta_0T}{\|b\|_f} +
 \frac{\sigma T^H}{\|b\|_{f}^{2}}(1 +\sigma T^H)\right)\right).
\end{eqnarray*}
Therefore, by assuming that there exists a known constant $\theta_{\max} > 0$ such that $\theta_0\in (0,\theta_{\max}]$, $T$ fulfills (\ref{consistency_computable_estimator_1}) when
\begin{displaymath}
T <\left(
\frac{\mathfrak c\ell^2}{\overline\alpha_H\sigma^2
\|b'\|_{\infty}^{2}}\exp\left(-2\mathfrak c_1\|b'\|_{\infty}\left(
\theta_{\max}^{2}T_{\max}^2 +
\frac{\theta_{\max}T_{\max}}{\ell} +
\frac{\sigma T_{\max}^{H}}{\ell^2}(1 +\sigma T_{\max}^{H})\right)\right)\right)^{\frac{1}{2H}}.
\end{displaymath}
\end{remark}
\noindent
Still when $\mathcal R_N =\emptyset$, the following proposition provides an asymptotic confidence interval for $\widetilde\theta_{N}^{\mathfrak c}$.
%


%
\begin{proposition}\label{ACI_fBm}
Assume that $b'(.)\leqslant 0$ and that $\theta_0 > 0$. If $\mathcal R_N =\emptyset$ and $T$ satisfies (\ref{consistency_computable_estimator_1}) with $\mathfrak c = 1/2$, then
\begin{displaymath}
\lim_{N\rightarrow\infty}
\mathbb P\left(\theta_0\in\left[
\widetilde\theta_{N}^{\mathfrak c =\frac{1}{2}} -\frac{2\overline Y_{N}^{1/2}}{\sqrt ND_N}u_{1 -\frac{\alpha}{4}}
\textrm{ $;$ }
\widetilde\theta_{N}^{\mathfrak c =\frac{1}{2}} +\frac{2\overline Y_{N}^{1/2}}{\sqrt ND_N}u_{1 -\frac{\alpha}{4}}\right]\right)\geqslant 1 -\alpha
\end{displaymath}
for every $\alpha\in (0,1)$, where
\begin{eqnarray*}
 \overline Y_N & := &
 \frac{\sigma^2}{NT^2}
 \sum_{i = 1}^{N}\left(
 \alpha_H\int_{[0,T]^2}|b(X_{s}^{i})|\cdot |b(X_{t}^{i})|\cdot |t - s|^{2H - 2}dsdt\right.\\
 & &
 \hspace{2cm}\left.
 +\alpha_{H}^{2}\sigma^2
 \int_{[0,T]^2}\int_{0}^{v}\int_{0}^{u}|u - u'|^{2H - 2}|v - v'|^{2H - 2}
 b'(X_{v}^{i})b'(X_{u}^{i})du'dv'dudv\right).
\end{eqnarray*}
\end{proposition}
%


%
\begin{proof}
The proof of Proposition \ref{ACI_fBm} is dissected in three steps.
\\
\\
{\bf Step 1.} Consider
\begin{displaymath}
U_N :=\frac{1}{N}\sum_{i = 1}^{N}Z^i,
\end{displaymath}
where $Z^1,\dots,Z^N$ are defined by
\begin{displaymath}
Z^i :=\frac{\sigma}{T}\int_{0}^{T}b(X_{s}^{i})\delta B_{s}^{i}
\textrm{ $;$ }\forall i\in\{1,\dots,N\}.
\end{displaymath}
First,
\begin{displaymath}
D_N =\frac{1}{NT}\sum_{i = 1}^{N}\int_{0}^{T}b(X_{s}^{i})^2ds
\xrightarrow[N\rightarrow\infty]{\mathbb P}
\mathbb E\left(\frac{1}{T}\int_{0}^{T}b(X_{s}^{1})^2ds\right) =\|b\|_{f}^{2} > 0
\end{displaymath}
by the (usual) law of large numbers, and
\begin{displaymath}
\sqrt NU_N\xrightarrow[N\rightarrow\infty]{\mathcal D}\mathcal N(0,{\rm var}(Z^1))
\end{displaymath}
by the (usual) central limit theorem. Then, since $dX_{t}^{i} =\theta_0b(X_{t}^{i})dt +\sigma dB_{t}^{i}$ for every $i\in\{1,\dots,N\}$, and by Slutsky's lemma,
\begin{displaymath}
\sqrt N(\widehat\theta_N -\theta_0) =
\sqrt N\frac{U_N}{D_N}
\xrightarrow[N\rightarrow\infty]{\mathcal D}
\mathcal N\left(0,\frac{{\rm var}(Z^1)}{\|b\|_{f}^{4}}\right).
\end{displaymath}
Now, consider
\begin{eqnarray*}
 \overline Y_{N}^{*} & := &
 \frac{\sigma^2}{NT^2}
 \sum_{i = 1}^{N}\left[
 \alpha_H\int_{[0,T]^2}b(X_{s}^{i})b(X_{t}^{i})|t - s|^{2H - 2}dsdt\right.\\
 & &
 \hspace{1.5cm}
 +\alpha_{H}^{2}\sigma^2
 \int_{[0,T]^2}\int_{0}^{v}\int_{0}^{u}|u - u'|^{2H - 2}|v - v'|^{2H - 2}\\
 & &
 \hspace{3cm}\left.\times
 b'(X_{v}^{i})b'(X_{u}^{i})
 \exp\left(\theta_0\left(\int_{u'}^{v}b'(X_{s}^{i})ds +\int_{v'}^{u}b'(X_{s}^{i})ds\right)\right)
 du'dv'dudv\right].
\end{eqnarray*}
By the law of large numbers, and by (\ref{isometry_divergence}),
\begin{eqnarray*}
 \overline Y_{N}^{*}
 & \xrightarrow[N\rightarrow\infty]{\mathbb P} &
 \frac{\sigma^2}{T^2}\mathbb E\left[
 \alpha_H\int_{[0,T]^2}b(X_{s}^{1})b(X_{t}^{1})|t - s|^{2H - 2}dsdt\right.\\
 & &
 \hspace{0.5cm}
 +\alpha_{H}^{2}\sigma^2
 \int_{[0,T]^2}\int_{0}^{v}\int_{0}^{u}|u - u'|^{2H - 2}|v - v'|^{2H - 2}\\
 & &
 \hspace{1cm}\left.\times
 b'(X_{v}^{1})b'(X_{u}^{1})
 \exp\left(\theta_0\left(\int_{u'}^{v}b'(X_{s}^{1})ds +\int_{v'}^{u}b'(X_{s}^{1})ds\right)\right)
 du'dv'dudv\right] = {\rm var}(Z^1).
\end{eqnarray*}
Therefore,
\begin{displaymath}
\frac{\sqrt ND_N}{|\overline Y_{N}^{*}|^{1/2}}(\widehat\theta_N -\theta_0)
\xrightarrow[N\rightarrow\infty]{\mathcal D}\mathcal N(0,1)
\end{displaymath}
by Slutsky's lemma, and then
\begin{displaymath}
\mathbb P\left(\frac{\sqrt ND_N}{|\overline Y_{N}^{*}|^{1/2}}|\widehat\theta_N -\theta_0| > x\right)
\xrightarrow[N\rightarrow\infty]{} 1 - (2\phi(x) - 1) = 2(1 -\phi(x))
\end{displaymath}
for every $x\in\mathbb R_+$.
\\
\\
{\bf Step 2.} Consider
\begin{displaymath}
\widehat\alpha_N :=\frac{\sqrt ND_N}{\overline Y_{N}^{1/2}}
\quad {\rm and}\quad
\widehat\alpha_{N}^{*} :=\frac{\sqrt ND_N}{|\overline Y_{N}^{*}|^{1/2}}.
\end{displaymath}
Since $b'(.)\leqslant 0$ and $\theta_0 > 0$, $\overline Y_{N}^{*}\leqslant\overline Y_N$, and then $\widehat\alpha_{N}^{*}\geqslant\widehat\alpha_N$. Moreover, as established in the proof of Proposition \ref{consistency_computable_estimator},
\begin{displaymath}
|\widetilde\theta_N -\widehat\theta_N|\leqslant
\frac{\mathfrak c}{1 -\mathfrak c}|\widehat\theta_N -\theta_0|
\quad {\rm on}\quad\Omega_N.
\end{displaymath}
So, by taking $\mathfrak c = 1/2$, for any $x\in\mathbb R_+$,
\begin{eqnarray*}
 \mathbb P(\widehat\alpha_N|\widetilde\theta_{N}^{\mathfrak c} -\theta_0| > 2x)
 & \leqslant &
 \mathbb P(\widehat\alpha_{N}^{*}|\widetilde\theta_{N}^{\mathfrak c} -\theta_0| > 2x)\\
 & \leqslant &
 \mathbb P(\widehat\alpha_{N}^{*}|\widetilde\theta_{N}^{\mathfrak c} -\widehat\theta_N| > x) +
 \mathbb P(\widehat\alpha_{N}^{*}|\widehat\theta_N -\theta_0| > x)\\
 & \leqslant &
 \mathbb P(\Omega_{N}^{c}) +
 \mathbb P(\{\widehat\alpha_{N}^{*}|\widehat\theta_N -
 \theta_0| > x\}\cap\Omega_N) +
 \mathbb P(\widehat\alpha_{N}^{*}|\widehat\theta_N -\theta_0| > x)\\
 & \leqslant &
 \mathbb P(\Omega_{N}^{c}) +
 2\mathbb P(\widehat\alpha_{N}^{*}|\widehat\theta_N -\theta_0| > x).
\end{eqnarray*}
Therefore,
\begin{displaymath}
\mathbb P(\widehat\alpha_N|\widetilde\theta_{N}^{\mathfrak c} -\theta_0|\leqslant 2x)
\geqslant 1 -\mathbb P(\Omega_{N}^{c}) -
2\mathbb P(\widehat\alpha_{N}^{*}|\widehat\theta_N -\theta_0| > x).
\end{displaymath}
{\bf Step 3 (conclusion).} By the two previous steps and since, as established in Proposition \ref{consistency_computable_estimator},
\begin{displaymath}
\lim_{N\rightarrow\infty}
\mathbb P(\Omega_{N}^{c}) = 0
\quad\textrm{under condition (\ref{consistency_computable_estimator_1})};
\end{displaymath}
for every $x\in\mathbb R_+$,
\begin{displaymath}
\lim_{N\rightarrow\infty}
\mathbb P(\widehat\alpha_N|\widetilde\theta_{N}^{\mathfrak c} -\theta_0|\leqslant 2x)
\geqslant 1 - 4(1 -\phi(x)) = 4\phi(x) - 3.
\end{displaymath}
So, for every $\alpha\in (0,1)$,
\begin{displaymath}
\lim_{N\rightarrow\infty}
\mathbb P(\widehat\alpha_N|\widetilde\theta_{N}^{\mathfrak c} -\theta_0|\leqslant 2u_{1 -\frac{\alpha}{4}})
\geqslant
4\phi(u_{1 -\frac{\alpha}{4}}) - 3 = 1 -\alpha.
\end{displaymath}
\end{proof}
\noindent
Now, even when $\mathcal R_N\neq\emptyset$, the following proposition provides a suitable risk bound on the truncated estimator
\begin{displaymath}
\widetilde\theta_{N}^{\mathfrak c,\mathfrak d} :=
\widetilde\theta_{N}^{\mathfrak c}\mathbf 1_{D_N\geqslant\mathfrak d}
\quad {\rm with}\quad
\mathfrak d\in\left(0,\frac{\|b\|_{f}^{2}}{2}\right].
\end{displaymath}
%


%
\begin{proposition}\label{risk_bound_computable_estimator}
Assume that $b'(.)\leqslant 0$, $\theta_0 > 0$ and that
\begin{equation}\label{risk_bound_computable_estimator_1}
\frac{2T^{2H}}{\|b\|_{f}^{2}}\exp\left(\frac{2\|b'\|_{\infty}}{\|b\|_{f}^{2}}
|\mathbb E({\tt b}(X_{T}^{1})) - {\tt b}(x_0)|\right)
<\frac{\mathfrak c}{\overline\alpha_H\sigma^2
\|b'\|_{\infty}^{2}}.
\end{equation}
Then, there exists a constant $\mathfrak c_{\ref{risk_bound_computable_estimator}} > 0$, not depending on $N$ and $\mathfrak d$, such that
\begin{displaymath}
\mathbb E[(\widetilde\theta_{N}^{\mathfrak c,\mathfrak d} -\theta_0)^2]
\leqslant
\frac{\mathfrak c_{\ref{risk_bound_computable_estimator}}}{\mathfrak d^2N}
\left(1 +\frac{|\mathcal R_N|}{N}\right).
\end{displaymath}
Moreover, if $|\mathcal R_N| = o(N^2)$, then both the truncated estimator $\widetilde\theta_{N}^{\mathfrak c,\mathfrak d}$ and the main estimator $\widetilde\theta_{N}^{\mathfrak c}$ are consistent.
\end{proposition}
%


%
\begin{proof}
First, as established in the proof of Proposition \ref{consistency_computable_estimator},
\begin{displaymath}
|\widetilde\theta_N -\theta_0|
\leqslant
\frac{1}{1 -\mathfrak c}
|\widehat\theta_N -\theta_0|
\quad {\rm on}\quad
\Omega_N.
\end{displaymath}
Then, since $\mathfrak c\in (0,1)$,
\begin{eqnarray*}
 |\widetilde\theta_{N}^{\mathfrak c,\mathfrak d} -\theta_0|
 & = &
 |\widetilde\theta_N -\theta_0|\mathbf 1_{\{D_N\geqslant\mathfrak d\}\cap\Omega_N} +
 |\theta_0|\mathbf 1_{(\{D_N\geqslant\mathfrak d\}\cap\Omega_N)^c}\\
 & \leqslant &
 (1 -\mathfrak c)^{-1}
 |\widehat\theta_N -\theta_0|\mathbf 1_{\{D_N\geqslant\mathfrak d\}\cap\Omega_N} +
 |\theta_0|(\mathbf 1_{D_N <\mathfrak d} +\mathbf 1_{\Omega_{N}^{c}})\\
 & \leqslant &
 (1 -\mathfrak c)^{-1}
 (|\widehat\theta_N -\theta_0|\mathbf 1_{D_N\geqslant\mathfrak d} +
 |\theta_0|\mathbf 1_{D_N <\mathfrak d}) +
 |\theta_0|\mathbf 1_{\Omega_{N}^{c}}\\
 & = &
 (1 -\mathfrak c)^{-1}
 |\widehat\theta_{N}^{\mathfrak d} -\theta_0| +
 |\theta_0|\mathbf 1_{\Omega_{N}^{c}}.
\end{eqnarray*}
So, by Proposition \ref{risk_bound_fBm},
\begin{displaymath}
\mathbb E[(\widetilde\theta_{N}^{\mathfrak c,\mathfrak d} -\theta_0)^2]
\leqslant
\frac{2}{(1 -\mathfrak c)^2}\cdot
\frac{\mathfrak c_{\ref{risk_bound_fBm}}}{N}\left(1 +\frac{|\mathcal R_N|}{N}\right) +
2\theta_{0}^{2}\mathbb P(\Omega_{N}^{c}).
\end{displaymath}
Now, let us show that $\mathbb P(\Omega_{N}^{c})$ is of order $1/N(1 + |\mathcal R_N|/N)^{-1}$. Consider $\mathfrak d_f :=\|b\|_{f}^{2}/2$,
\begin{displaymath}
\mathfrak c_1 :=
\frac{\mathfrak c}{\overline\alpha_HT^{2H}
\sigma^2\|b'\|_{\infty}^{2}},
\end{displaymath}
and note that
\begin{eqnarray*}
 \mathbb P(\Omega_{N}^{c}) & = &
 \mathbb P\left(\frac{M_N}{D_N} >\mathfrak c_1,D_N\geqslant\mathfrak d_f\right) +
 \mathbb P\left(\frac{M_N}{D_N} >\mathfrak c_1,D_N <\mathfrak d_f\right)\\
 & \leqslant &
 \mathbb P\left(\frac{M_N}{\mathfrak d_f} >\mathfrak c_1,D_N\geqslant\mathfrak d_f\right) +
 \mathbb P(D_N <\mathfrak d_f).
\end{eqnarray*}
On the one hand, as established in the proof of Proposition \ref{risk_bound_BM},
\begin{eqnarray*}
 \mathbb P(D_N <\mathfrak d_f)
 & \leqslant &
 \mathbb P\left(|D_N -\|b\|_{f}^{2}| >\frac{\|b\|_{f}^{2}}{2}\right)\\
 & \leqslant &
 \frac{1}{\mathfrak d_{f}^{2}}\mathbb E[(D_N -\|b\|_{f}^{2})^2]
 \leqslant\frac{\|b^2\|_{f}^{2}}{\mathfrak d_{f}^{2}N}\left(1 +\frac{|\mathcal R_N|}{N}\right).
\end{eqnarray*}
On the other hand,
\begin{eqnarray*}
 \mathbb P\left(\frac{M_N}{\mathfrak d_f} >\mathfrak c_1,D_N\geqslant\mathfrak d_f\right)
 & = &
 \mathbb P\left[
 \frac{1}{N}\left|\sum_{i = 1}^{N}({\tt b}(X_{T}^{i}) - {\tt b}(x_0))\right| >
 \log\left(\frac{\mathfrak c}{\overline\alpha_HT^{2H}
 \sigma^2\|b'\|_{\infty}^{2}}\mathfrak d_f\right)\frac{D_N}{\|b'\|_{\infty}},
 D_N\geqslant\mathfrak d_f\right]\\
 & \leqslant &
 \mathbb P\left[
 \frac{1}{N}\left|\sum_{i = 1}^{N}({\tt b}(X_{T}^{i}) - {\tt b}(x_0))\right| >
 \log\left(\frac{\mathfrak c}{\overline\alpha_HT^{2H}
 \sigma^2\|b'\|_{\infty}^{2}}\mathfrak d_f\right)\frac{\mathfrak d_f}{\|b'\|_{\infty}}\right]\\
 & \leqslant &
 \mathbb P\left[
 \frac{1}{N}\left|\sum_{i = 1}^{N}[
 {\tt b}(X_{T}^{i}) - {\tt b}(x_0) -
 (\mathbb E({\tt b}(X_{T}^{i})) - {\tt b}(x_0))]\right| >\mathfrak u\right]
\end{eqnarray*}
with
\begin{displaymath}
\mathfrak u =
\log\left(\frac{\mathfrak c}{\overline\alpha_HT^{2H}
\sigma^2\|b'\|_{\infty}^{2}}\mathfrak d_f\right)\frac{\mathfrak d_f}{\|b'\|_{\infty}} -
|\mathbb E({\tt b}(X_{T}^{1})) - {\tt b}(x_0)| > 0
\quad {\rm by}\quad (\ref{risk_bound_computable_estimator_1}).
\end{displaymath}
So, by the Bienaym\'e-Tchebychev inequality,
\begin{eqnarray*}
 \mathbb P\left(\frac{M_N}{\mathfrak d_f} >\mathfrak c_1,D_N\geqslant\mathfrak d_f\right)
 & \leqslant &
 \frac{1}{\mathfrak u^2N^2}{\rm var}\left(\sum_{i = 1}^{N}
 ({\tt b}(X_{T}^{i}) - {\tt b}(x_0))\right)\\
 & = &
 \frac{1}{\mathfrak u^2N}{\rm var}(\texttt b(X_{T}^{1}) -\texttt b(x_0))\\
 & &
 \hspace{2cm} +
 \frac{1}{\mathfrak u^2N^2}\sum_{(i,k)\in\mathcal R_N}
 {\rm cov}(\texttt b(X_{T}^{i}) -\texttt b(x_0),
 \texttt b(X_{T}^{k}) -\texttt b(x_0))\\
 & \leqslant &
 \frac{1}{\mathfrak u^2N}\left(1 +\frac{|\mathcal R_N|}{N}\right)
 \mathbb E[({\tt b}(X_{T}^{1}) - {\tt b}(x_0))^2].
\end{eqnarray*}
Therefore,
\begin{displaymath}
\mathbb E[(\widetilde\theta_{N}^{\mathfrak c,\mathfrak d} -\theta_0)^2]
\leqslant
\left[\frac{\mathfrak c_{\ref{risk_bound_fBm}}}{(1 -\mathfrak c)^2} +
\theta_{0}^{2}\left(\frac{\|b^2\|_{f}^{2}}{\mathfrak d_{f}^{2}} +
\frac{1}{\mathfrak u^2}
\mathbb E[({\tt b}(X_{T}^{1}) - {\tt b}(x_0))^2]\right)\right]
\frac{2}{N}\left(1 +\frac{|\mathcal R_N|}{N}\right).
\end{displaymath}
Finally, assume that $|\mathcal R_N| = o(N^2)$. By following the same line than in the second step of the proof of Proposition \ref{risk_bound_BM}, thanks to the bounds on
\begin{displaymath}
\mathbb E[(D_N -\|b\|_{f}^{2})^2]
\quad {\rm and}\quad
\mathbb E[(\widetilde\theta_{N}^{\mathfrak c,\mathfrak d} -\theta_0)^2]
\end{displaymath}
previously established, $\widetilde\theta_{N}^{\mathfrak c,\mathfrak d}$ and $\widetilde\theta_{N}^{\mathfrak c}$ are consistent.
\end{proof}
%


%
\begin{remark}\label{remark_threshold_fBm_part_2}
Let us make some remarks about Proposition \ref{risk_bound_computable_estimator}:
\begin{enumerate}
 \item Note that since $\|b\|_{f}^{2}/2 <\|b\|_{f}^{2}$, if (\ref{risk_bound_computable_estimator_1}) is fulfilled, then (\ref{consistency_computable_estimator_1}) is fulfilled too. Moreover, Remark \ref{remark_condition_T} may be easily adapted from (\ref{consistency_computable_estimator_1}) to (\ref{risk_bound_computable_estimator_1}).
 \item As for $H = 1/2$, by (the second part of) Proposition \ref{risk_bound_computable_estimator}, $\widetilde\theta_{N}^{\mathfrak c}$ and $\widetilde\theta_{N}^{\mathfrak c,\mathfrak d}$ are both consistent, which means that for large values of $N$, one should always consider our main estimator $\widetilde\theta_{N}^{\mathfrak c}$, especially when $\mathcal R_N =\emptyset$ because Proposition \ref{ACI_fBm} provides (in addition) an asymptotic confidence interval for $\widetilde\theta_{N}^{\mathfrak c}$. In this last situation, as usual, $N\geqslant 30$ may be considered as large. However, there are non-asymptotic theoretical guarantees on the truncated estimator $\widetilde\theta_{N}^{\mathfrak c,\mathfrak d}$ thanks to the risk bound provided in (the first part of) Proposition \ref{risk_bound_fBm}, but there are no such guarantees on $\widetilde\theta_{N}^{\mathfrak c}$. So, as for $H = 1/2$, in the non-asymptotic framework, the following rule should be observed to choose the threshold $\mathfrak d$ in $\Delta_f$: in order to degrade as few as possible the theoretical guarantees on $\widetilde\theta_{N}^{\mathfrak c,\mathfrak d}$ provided by the risk bound in Proposition \ref{risk_bound_computable_estimator}, one should take $\mathfrak d =\mathfrak d_{\max}$, where $\mathfrak d_{\max}$ is the highest (easily) computable lower-bound on $\|b\|_{f}^{2}/2$. This doesn't mean that smaller values of $\mathfrak d$ cannot give better numerical results, but then theoretical guarantees of Proposition \ref{risk_bound_computable_estimator} are degraded. See Remark \ref{remark_threshold_fBm_part_1} on how to find a computable lower-bound $\mathfrak d_{\max}$ on $\|b\|_{f}^{2}/2$ in some situations.
\end{enumerate}
\end{remark}
\noindent
Finally, let us consider the estimators
\begin{displaymath}
\widetilde\theta_{N,n}^{\mathfrak c} :=\widetilde\theta_{N,n}\mathbf 1_{D_N\geqslant\mathfrak d}
\quad {\rm and}\quad
\widetilde\theta_{N,n}^{\mathfrak c,\mathfrak d} :=
\widetilde\theta_{N,n}\mathbf 1_{\{D_N\geqslant\mathfrak d\}\cap\Omega_N}
\quad\textrm{$;$}\quad n\in\mathbb N,
\end{displaymath}
where $\widetilde\theta_{N,n} := I_N + R_{N,n}$, and the sequence $(R_{N,n})_{n\in\mathbb N}$ is defined by $R_{N,0} = 0$ and
\begin{displaymath}
R_{N,n + 1} =\Phi_N(R_{N,n})
\textrm{ $;$ }n\in\mathbb N.
\end{displaymath}
%


%
\begin{proposition}\label{convergence_approximate_estimator}
Let $\psi :\mathbb N\rightarrow\mathbb N$ be a map satisfying
\begin{displaymath}
\psi(.)\geqslant
-\frac{\log(\mathfrak m_{\ref{convergence_approximate_estimator}}
\sqrt .)}{\log(\mathfrak c)}
\quad\textrm{with}\quad
\mathfrak m_{\ref{convergence_approximate_estimator}} =
\frac{\mathfrak c(1 -\mathfrak c)^{-1}}{2T\overline\alpha_H\|b'\|_{\infty}}.
\end{displaymath}
\begin{enumerate}
 \item If $\mathcal R_N =\emptyset$ and $T$ satisfies (\ref{consistency_computable_estimator_1}), then
 \begin{displaymath}
 \widetilde\theta_{N,\psi(N)}^{\mathfrak c}
 \xrightarrow[N\rightarrow\infty]{\mathbb P}\theta_0.
 \end{displaymath}
 \item Assume that $T$ satisfies (\ref{risk_bound_computable_estimator_1}). Then, there exists a constant $\mathfrak c_{\ref{convergence_approximate_estimator}} > 0$, not depending on $N$ and $\mathfrak d$, such that
 \begin{displaymath}
 \mathbb E[(\widetilde\theta_{N,\psi(N)}^{\mathfrak c,\mathfrak d} -\theta_0)^2]
 \leqslant
 \frac{\mathfrak c_{\ref{convergence_approximate_estimator}}}{\mathfrak d^2N}
 \left(1 +\frac{|\mathcal R_N|}{N}\right).
 \end{displaymath}
 Moreover, if $|\mathcal R_N| = o(N^2)$, then both the truncated estimator $\widetilde\theta_{N,\psi(N)}^{\mathfrak c,\mathfrak d}$ and the main estimator $\widetilde\theta_{N,\psi(N)}^{\mathfrak c}$ are consistent.
\end{enumerate}
\end{proposition}
%


%
\begin{proof}
On the event $\Omega_N$, note that
\begin{eqnarray*}
 |\Phi_N(0)| & \leqslant &
 \frac{\alpha_H\sigma^2}{NTD_N}\sum_{i = 1}^{N}
 \int_{0}^{T}\int_{0}^{t}|b'(X_{t}^{i})|
 \exp\left(I_N\int_{s}^{t}b'(X_{u}^{i})du\right)|t - s|^{2H - 2}dsdt\\
 & \leqslant &
 \frac{\alpha_H\sigma^2}{TD_N}
 \|b'\|_{\infty}M_N
 \int_{0}^{T}\int_{0}^{t}|t - s|^{2H - 2}dsdt
 \leqslant
 \frac{\sigma^2\|b'\|_{\infty}}{2T}
 T^{2H}\frac{M_N}{D_N}
 \leqslant
 \mathfrak c_1\quad {\rm with}\quad
 \mathfrak c_1 =
 \frac{\mathfrak c}{2T\overline\alpha_H\|b'\|_{\infty}}.
\end{eqnarray*}
Consider $n\in\mathbb N^*$. Thanks to a well-known consequence of Picard's fixed point theorem, for every $x\in\mathbb R_+$,
\begin{displaymath}
|(\underbrace{\Phi_N\circ\dots\circ\Phi_N}_{\textrm{$n$ times}})(x) - R_N|
\leqslant\frac{\mathfrak c^n}{1 -\mathfrak c}|\Phi_N(x) - x|
\quad {\rm on}\quad\Omega_N.
\end{displaymath}
So,
\begin{eqnarray*}
 |R_{N,n} - R_N|\mathbf 1_{\Omega_N}
 & = &
 |(\Phi_N\circ\dots\circ\Phi_N)(R_{N,0}) - R_N|\mathbf 1_{\Omega_N}\\
 & \leqslant &
 \frac{\mathfrak c^n}{1 -\mathfrak c}|\Phi_N(0)|\mathbf 1_{\Omega_N}
 \leqslant
 \mathfrak c_2\mathfrak c^n
 \quad {\rm with}\quad
 \mathfrak c_2 =
 \frac{\mathfrak c_1}{1 -\mathfrak c}.
\end{eqnarray*}
First, if $\mathcal R_N =\emptyset$ and $T$ satisfies (\ref{consistency_computable_estimator_1}), then
\begin{eqnarray*}
 |\widetilde\theta_{N,\psi(N)}^{\mathfrak c} -\theta_0|
 & \leqslant &
 |R_{N,\psi(N)} - R_N|\mathbf 1_{\Omega_N} +
 |\widetilde\theta_{N}^{\mathfrak c} -\theta_0|\\
 & \leqslant &
 \frac{1}{\sqrt N} +
 |\widetilde\theta_{N}^{\mathfrak c} -\theta_0|
 \xrightarrow[N\rightarrow\infty]{\mathbb P} 0
\end{eqnarray*}
by Proposition \ref{consistency_computable_estimator}. Now, assume that $T$ satisfies (\ref{risk_bound_computable_estimator_1}). One the one hand,
\begin{eqnarray*}
 \mathbb E[(\widetilde\theta_{N,\psi(N)}^{\mathfrak c,\mathfrak d} -\theta_0)^2]
 & \leqslant &
 2\mathbb E[(R_{N,\psi(N)} - R_N)^2\mathbf 1_{\{D_N\geqslant\mathfrak d\}\cap\Omega_N}] +
 2\mathbb E[(\widetilde\theta_{N}^{\mathfrak c,\mathfrak d} -\theta_0)^2]\\
 & \leqslant &
 \frac{2(1 +\mathfrak c_{\ref{risk_bound_computable_estimator}}\mathfrak d^{-2})}{N}
 \left(1 +\frac{|\mathcal R_N|}{N}\right)
\end{eqnarray*}
by Proposition \ref{risk_bound_computable_estimator}. On the other hand, if $|\mathcal R_N| = o(N^2)$, by following the same line than in the second step of the proof of Proposition \ref{risk_bound_BM}, thanks to the bounds on
\begin{displaymath}
\mathbb E[(D_N -\|b\|_{f}^{2})^2]
\quad {\rm and}\quad
\mathbb E[(\widetilde\theta_{N,\psi(N)}^{\mathfrak c,\mathfrak d} -\theta_0)^2]
\end{displaymath}
previously established, $\widetilde\theta_{N,\psi(N)}^{\mathfrak c,\mathfrak d}$ and $\widetilde\theta_{N,\psi(N)}^{\mathfrak c}$ are consistent.
\end{proof}
%


%
\section{A few elements in the case $H < 1/2$}\label{section_fBm_rough}
Throughout this section, $H\in (1/3,1/2)$ and
\begin{equation}\label{SDE_fBm_rough}
X_{t}^{i} = x_0 +\theta_0\int_{0}^{t}b(X_{s}^{i})ds +\sigma B_{t}^{i}
\textrm{ $;$ }t\in [0,T]\textrm{, }i\in\{1,\dots,N\},
\end{equation}
where $b\in C^2(\mathbb R)$, $b'$ and $b''$ are bounded, and $\sigma\neq 0$. Under these conditions, Equation (\ref{SDE_fBm_rough}) has a unique (pathwise) solution $(X_{t}^{1},\dots,X_{t}^{N})_{t\in [0,T]}$. First, as for $H > 1/2$, the (usual) law of large numbers together with Equality (\ref{zero_mean_Skorokhod_integral}) ensure that the auxiliary estimator
\begin{displaymath}
\widehat\theta_N :=
\left(\sum_{i = 1}^{N}\int_{0}^{T}b(X_{s}^{i})^2ds\right)^{-1}
\left(\sum_{i = 1}^{N}\int_{0}^{T}b(X_{s}^{i})\delta X_{s}^{i}\right)
\end{displaymath}
is consistent when
\begin{eqnarray*}
 \mathcal R_N & := &
 \{(i,k)\in\{1,\dots,N\}^2 : i\neq k
 \textrm{ and $X^i$ is not independent of $X^k$}\}\\
 & = &
 \{(i,k)\in\{1,\dots,N\}^2 : i\neq k
 \textrm{ and $B^i$ is not independent of $B^k$}\}\\
 & = &
 \{(i,k)\in\{1,\dots,N\}^2 : i\neq k\textrm{ and }R_{i,k}\neq 0\} =\emptyset,
\end{eqnarray*}
and the following proposition provides a suitable risk bound on the truncated auxiliary estimator
\begin{displaymath}
\widehat\theta_{N}^{\mathfrak d} :=
\widehat\theta_N\mathbf 1_{D_N\geqslant\mathfrak d}
\quad {\rm with}\quad
\mathfrak d\in\Delta_f =\left(0,\frac{\|b\|_{f}^{2}}{2}\right]
\end{displaymath}
thanks to Hu et al. \cite{HNZ19}, Proposition 4.4.(1).
%


%
\begin{proposition}\label{risk_bound_fBm_rough}
If $\theta_0 > 0$ and $b$ satisfies the dissipativity condition (\ref{dissipativity_condition}), then
\begin{displaymath}
\mathbb E[(\widehat\theta_{N}^{\mathfrak d} -\theta_0)^2]
\leqslant\frac{\mathfrak c_{\ref{risk_bound_fBm_rough}}}{N}
\left(1 +\frac{|\mathcal R_N|}{N}\right),
\end{displaymath}
where
\begin{displaymath}
\mathfrak c_{\ref{risk_bound_fBm_rough}} :=
\frac{1}{\mathfrak d^2}\left(
\sigma^2\overline{\mathfrak c}_{\ref{risk_bound_fBm_rough}}
\frac{1 + T^{2H} + T^{4H}}{T^{2 - 2H}} +
\theta_{0}^{2}\|b^2\|_{f}^{2}\right)
\end{displaymath}
and $\overline{\mathfrak c}_{\ref{risk_bound_fBm_rough}}$ is a positive constant not depending $N$ and $T$.
\end{proposition}
%


%
\begin{proof}
First of all, since $dX_{t}^{i} =\theta_0b(X_{t}^{i})dt +\sigma dB_{t}^{i}$ for every $i\in\{1,\dots,N\}$,
\begin{displaymath}
\widehat\theta_N =
\theta_0 +\frac{U_N}{D_N}
\quad {\rm with}\quad
U_N =
\frac{\sigma}{NT}\sum_{i = 1}^{N}\int_{0}^{T}b(X_{s}^{i})\delta B_{s}^{i}.
\end{displaymath}
By Cauchy-Schwarz's inequality and Hu et al. \cite{HNZ19}, Proposition 4.4.(1), there exists a constant $\mathfrak c_1 > 0$, not depending on $N$ and $T$, such that
\begin{eqnarray*}
 \mathbb E(U_{N}^{2}) & = &
 \frac{\sigma^2}{N^2T^2}
 \sum_{i = 1}^{N}\mathbb E\left[
 \left(\int_{0}^{T}b(X_{s}^{i})\delta B_{s}^{i}\right)^2\right]\\
 & &
 \hspace{2cm} +
 \frac{\sigma^2}{N^2T^2}\sum_{(i,k)\in\mathcal R_N}
 \mathbb E\left(\left(\int_{0}^{T}b(X_{s}^{i})\delta B_{s}^{i}\right)
 \left(\int_{0}^{T}b(X_{s}^{k})\delta B_{s}^{k}\right)\right)\\
 & \leqslant &
 \underbrace{\sigma^2\mathfrak c_1
 \frac{1 + T^{2H} + T^{4H}}{T^{2 - 2H}}}_{=:\mathfrak c_2}
 \cdot\frac{1}{N}\left(1 +\frac{|\mathcal R_N|}{N}\right).
\end{eqnarray*}
As in the proof of Proposition \ref{risk_bound_BM},
\begin{displaymath}
\mathbb E[(D_N -\|b\|_{f}^{2})^2]
\leqslant
\frac{\|b^2\|_{f}^{2}}{N}\left(1 +\frac{|\mathcal R_N|}{N}\right),
\end{displaymath}
and then
\begin{eqnarray*}
 \mathbb E[(\widehat\theta_{N}^{\mathfrak d} -\theta_0)^2]
 & \leqslant &
 \frac{1}{\mathfrak d^2}[\mathbb E(U_{N}^{2}) +
 \theta_{0}^{2}\mathbb E[(D_N -\|b\|_{f}^{2})^2]]\\
 & \leqslant &
 \frac{\mathfrak c_3}{N}\left(1 +\frac{|\mathcal R_N|}{N}\right)
 \quad {\rm with}\quad
 \mathfrak c_3 =
 \frac{\mathfrak c_2 +\theta_{0}^{2}\|b^2\|_{f}^{2}}{\mathfrak d^2}.
\end{eqnarray*}
\end{proof}
\noindent
Now, let us consider $p\in (1/H,3)$, $\Delta_T :=\{(s,t)\in [0,T]^2 : s < t\}$, $i\in\{1,\dots,N\}$ and the enhanced fractional Brownian motion $\mathbf B^i = (B^i,\mathbb B^i)$, where $\mathbb B^i = (\mathbb B_{s,t}^{i})_{(s,t)\in\Delta_T}$ is the stochastic process defined by
\begin{displaymath}
\mathbb B_{s,t}^{i} :=
\int_{s < r_1 < r_2 < t}dB_{r_1}^{i}dB_{r_2}^{i} =
\frac{(B_{t}^{i} - B_{s}^{i})^2}{2}
\textrm{ $;$ }\forall (s,t)\in\Delta_T.
\end{displaymath}
For every $(s,t)\in\Delta_T$,
\begin{displaymath}
X_{t}^{i} - X_{s}^{i} = (X_{s}^{i})'(B_{t}^{i} - B_{s}^{i}) +\texttt R_{t}^{i} -\texttt R_{s}^{i},
\end{displaymath}
where $(X^i)'\equiv\sigma$ is the Gubinelli derivative of $X^i$, and
\begin{displaymath}
\texttt R^i : t\in [0,T]\longmapsto
\theta_0\int_{0}^{t}b(X_{s}^{i})ds
\end{displaymath}
is a continuous function of finite $1$-variation (i.e.
\begin{displaymath}
\|\texttt R^i\|_{1\textrm{-var},T} :=
\sup_{(t_j)\textrm{ \tiny dissection of }[0,T]}\sum_j|\texttt R_{t_{j + 1}}^{i} -\texttt R_{t_j}^{i}| <\infty).
\end{displaymath}
In other words, the paths of $X^i$ are controlled by those of $B^i$ (see Friz \& Hairer \cite{FH14}, Definition 4.6). As established in Friz \& Hairer \cite{FH14}, Section 7.3, there exists a continuous function $\texttt R_{b}^{i} : [0,T]\rightarrow\mathbb R$, of finite $p/2$-variation (i.e.
\begin{displaymath}
\|\texttt R_{b}^{i}\|_{\frac{p}{2}\textrm{-var},T} :=
\sup_{(t_j)\textrm{ \tiny dissection of }[0,T]}\left(
\sum_j|\texttt R_{b}^{i}(t_{j + 1}) -\texttt R_{b}^{i}(t_j)|^{\frac{p}{2}}\right)^{\frac{2}{p}} <\infty),
\end{displaymath}
such that
\begin{displaymath}
b(X_{t}^{i}) - b(X_{s}^{i}) = b(X^i)_s'(B_{t}^{i} - B_{s}^{i}) +\texttt R_{b}^{i}(t) -\texttt R_{b}^{i}(s)
\textrm{ $;$ }\forall (s,t)\in\Delta_T
\end{displaymath}
with
\begin{displaymath}
b(X^i)_s' = b'(X_{s}^{i})(X_{s}^{i})' =\sigma b'(X_{s}^{i})
\end{displaymath}
for every $s\in [0,T]$. So, the paths of $Y^i = (b(X_{s}^{i}))_{s\in [0,T]}$ are also controlled by those of $B^i$, and then the pathwise integral (in the sense of rough paths) of $Y^i$ with respect to $\mathbf B^i$ on $[0,T]$ is well-defined (see Friz \& Hairer \cite{FH14}, Theorem 4.10):
\begin{displaymath}
\int_{0}^{T}b(X_{s}^{i})d\mathbf B_{s}^{i} :=
\lim_{\pi(D)\rightarrow 0}\sum_{[u,v]\in D}(b(X_{u}^{i})(B_{v}^{i} - B_{u}^{i}) +
b(X^i)_u'\mathbb B_{u,v}^{i})
\quad\textrm{ for any dissection $D$ of $[0,T]$.}
\end{displaymath}
Moreover, assume that
\begin{equation}\label{conditions_Song_Tindel}
\mathbb E(\|\mathbf D_{0}^{i}Y^i\|_{p\textrm{-var},T}^{2}) <\infty
\quad {\rm and}\quad
\mathbb E(\|\mathbf D^iY^i\|_{p\textrm{-var},[0,T]^2}^{2}) <\infty,
\end{equation}
and let us recall that $R$ is the covariance function of the fractional Brownian motion. Under condition (\ref{conditions_Song_Tindel}), by Song \& Tindel \cite{ST22}, Theorem 3.1,
\begin{eqnarray*}
 \int_{0}^{T}b(X_{s}^{i})\delta B_{s}^{i}
 & = &
 \int_{0}^{T}b(X_{s}^{i})d\mathbf B_{s}^{i} -\frac{1}{2}\int_{0}^{T}b(X^i)_s'dV(s) -
 \int_{0 < t_1 < t_2 < T}
 (\mathbf D_{t_1}^{i}Y_{t_2}^{i} - b(X^i)_{t_2}')dR(t_1,t_2)\\
 & = &
 \int_{0}^{T}b(X_{s}^{i})d\mathbf B_{s}^{i} -\frac{\sigma}{2}\int_{0}^{T}b'(X_{s}^{i})dV(s)\\
 & &
 \hspace{1.5cm} -
 \sigma\int_{0 < t_1 < t_2 < T}
 \left(b'(X_{t_2}^{i})\exp\left(\theta_0\int_{t_1}^{t_2}b'(X_{u}^{i})du\right) -
 b'(X_{t_2}^{i})\right)dR(t_1,t_2),
\end{eqnarray*}
where $V(t) := R(t,t)$ for every $t\in [0,T]$. Then, as for $H > 1/2$, since $\widehat\theta_N$ is a converging estimator of $\theta_0$,
\begin{displaymath}
\widehat\theta_N - I_N =\Phi_N(\theta_0 - I_N)
\approx\Phi_N(\widehat\theta_N - I_N),
\end{displaymath}
where
\begin{eqnarray*}
 \Phi_N(r)
 & := &
 -\frac{\sigma^2}{NTD_N}\left[
 \frac{1}{2}\sum_{i = 1}^{N}\int_{0}^{T}b'(X_{s}^{i})dV(s)\right.\\
 & &
 \hspace{1cm}\left. +
 \sum_{i = 1}^{N}
 \int_{0 < t_1 < t_2 < T}
 \left(b'(X_{t_2}^{i})\exp\left((r + I_N)\int_{t_1}^{t_2}b'(X_{u}^{i})du\right) -
 b'(X_{t_2}^{i})\right)dR(t_1,t_2)\right]
\end{eqnarray*}
and, by the change of variable formula for the rough integral,
\begin{eqnarray*}
 I_N & := &
 \frac{1}{NTD_N}\sum_{i = 1}^{N}\int_{0}^{T}b(X_{s}^{i})dX_{s}^{i}\\
 & = &
 \frac{1}{NTD_N}\sum_{i = 1}^{N}(\texttt b(X_{T}^{i}) -\texttt b(x_0))
 \quad {\rm with}\quad
 \texttt b'(.) = b(.).
\end{eqnarray*}
This legitimates to consider the estimator $\widetilde\theta_N := I_N + R_N$ of $\theta_0$, where $R_N$ is the fixed point (when it exists and is unique) of the map $\Phi_N$. Unfortunately, up to our knowledge, the only control of the 2D Young integral (see Friz \& Victoir \cite{FV10}, Sections 5.5 and 6.4) involved in the definition of $\Phi_N$ is the Young-L\'oeve-Towghi inequality (see Friz and Victoir \cite{FV10}, Theorem 6.18), which seems not sharp enough in order to establish the existence and the uniqueness of $\widetilde\theta_N$ under a condition as simple as (\ref{existence_uniqueness_R_N_1}). This difficulty deserves further investigations, but it's out of the scope of the present paper.
%


%
\section{Numerical experiments}\label{section_numerical_experiments}
In this section, $H\in (1/2,1)$, and our computable estimator of $\theta_0$ is evaluated on the two following models:
\begin{enumerate}
 \item $dX_t = \theta_0b_1(X_t)dt + 0.25dB_t$ with $T = 0.1$, $x_0 = 5$, $\theta_0 = 1$ and $b_1 :=\pi -\arctan$.
 \item $dX_t = \theta_0b_2(X_t)dt + dB_t$ with $T = 0.75$, $x_0 = 5$, $\theta_0 = 1$ and $b_2 := -{\rm Id}_{\mathbb R}$.
\end{enumerate}
Both Model 1 and Model 2 are simulated via the Euler method along the dissection $(t_0,\dots,t_{\nu})$ of constant mesh $T/\nu$ with $\nu = 20$. Moreover, the map $\Phi_N$ and the bounds of the asymptotic confidence interval (ACI) in Proposition \ref{ACI_fBm} are both approximated by replacing all (usual) integrals by the corresponding Riemann's sums along the dissection $(t_0,\dots,t_{\nu})$ in their definitions. According to Proposition \ref{convergence_approximate_estimator}, the fixed point $R_N$ of $\Phi_N$ is approximated by $(\Phi_N\circ\dots\circ\Phi_N)(0)$ ($n = 30$ times).
\\
\\
First, for each model, with $H = 0.7$ and $H = 0.9$, $\widetilde\theta_N$ is computed from $N = 1,\dots,50$ paths of the process $X$. This experiment is repeated $100$ times. The means and the standard deviations of the error $|\widetilde\theta_{50} -\theta_0|$ are stored in Table \ref{table_errors}.
\begin{table}[!h]
\begin{center}
\begin{tabular}{|l||c|c|}
 \hline
  & Mean error & Error StD.\\
 \hline
 \hline
 Model 1 ($H = 0.7$) & 0.0386852 & 0.0171819\\
 \hline
 Model 1 ($H = 0.9$) & 0.0138193 & 0.0155206\\
 \hline
 \hline
 Model 2 ($H = 0.7$) & 0.0489564 & 0.0335672\\
 \hline
 Model 2 ($H = 0.9$) & 0.0186479 & 0.0139259\\
 \hline
\end{tabular}
\medskip
\caption{Means and StD. of the error of $\widetilde\theta_{50}$ (100 repetitions).}\label{table_errors}
\end{center}
\end{table}
The mean error of $\widetilde\theta_{50}$ is small in the four situations: lower than $4.9\cdot 10^{-2}$. Since the error standard deviation of $\widetilde\theta_{50}$ is also small in each situations ($< 3.4\cdot 10^{-2}$), on average, the error of $\widetilde\theta_{50}$ for one repetition of the experiment should be near to its mean error. Note also that for both Model 1 and Model 2, the mean error of $\widetilde\theta_{50}$ is higher when $H = 0.7$ than when $H = 0.9$; probably because $H$ controls the H\"older exponent of the paths of the fractional Brownian motion. In Figure \ref{plots_Model_1} (resp. Figure \ref{plots_Model_2}), for $N = 1,\dots,50$, $\widetilde\theta_N$ and the bounds of the $95\%$-ACI are plotted for one of the 100 datasets generated from Model 1 (resp. Model 2). These figures illustrate both that $\widetilde\theta_N\rightarrow\theta_0$ and the rate of convergence.
\begin{figure}[H]
\begin{minipage}[b]{0.4\linewidth}
 \centering
 \includegraphics[width=7cm,height=6cm]{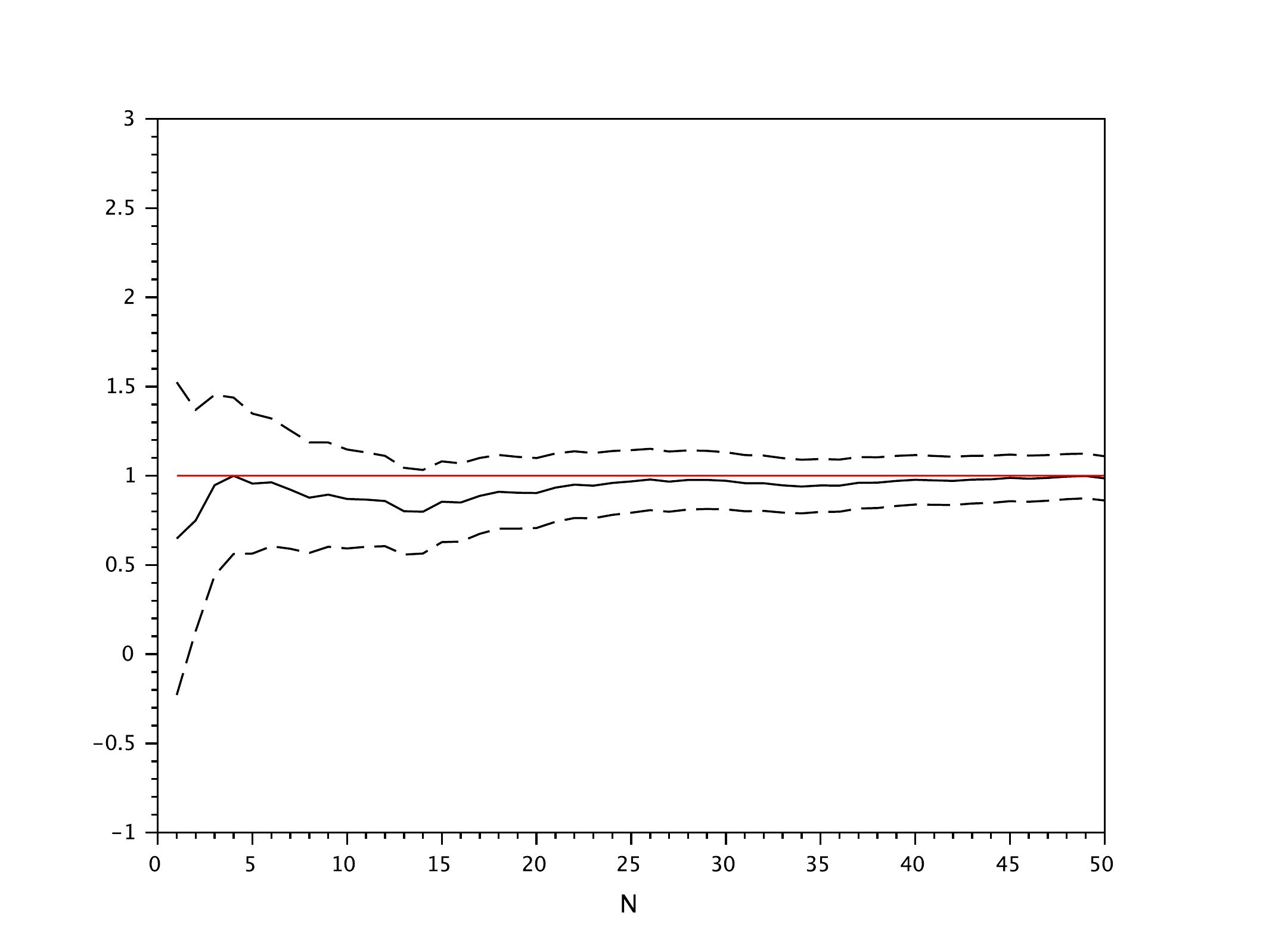}
\end{minipage}
\begin{minipage}[b]{0.4\linewidth}
 \centering
 \includegraphics[width=7cm,height=6cm]{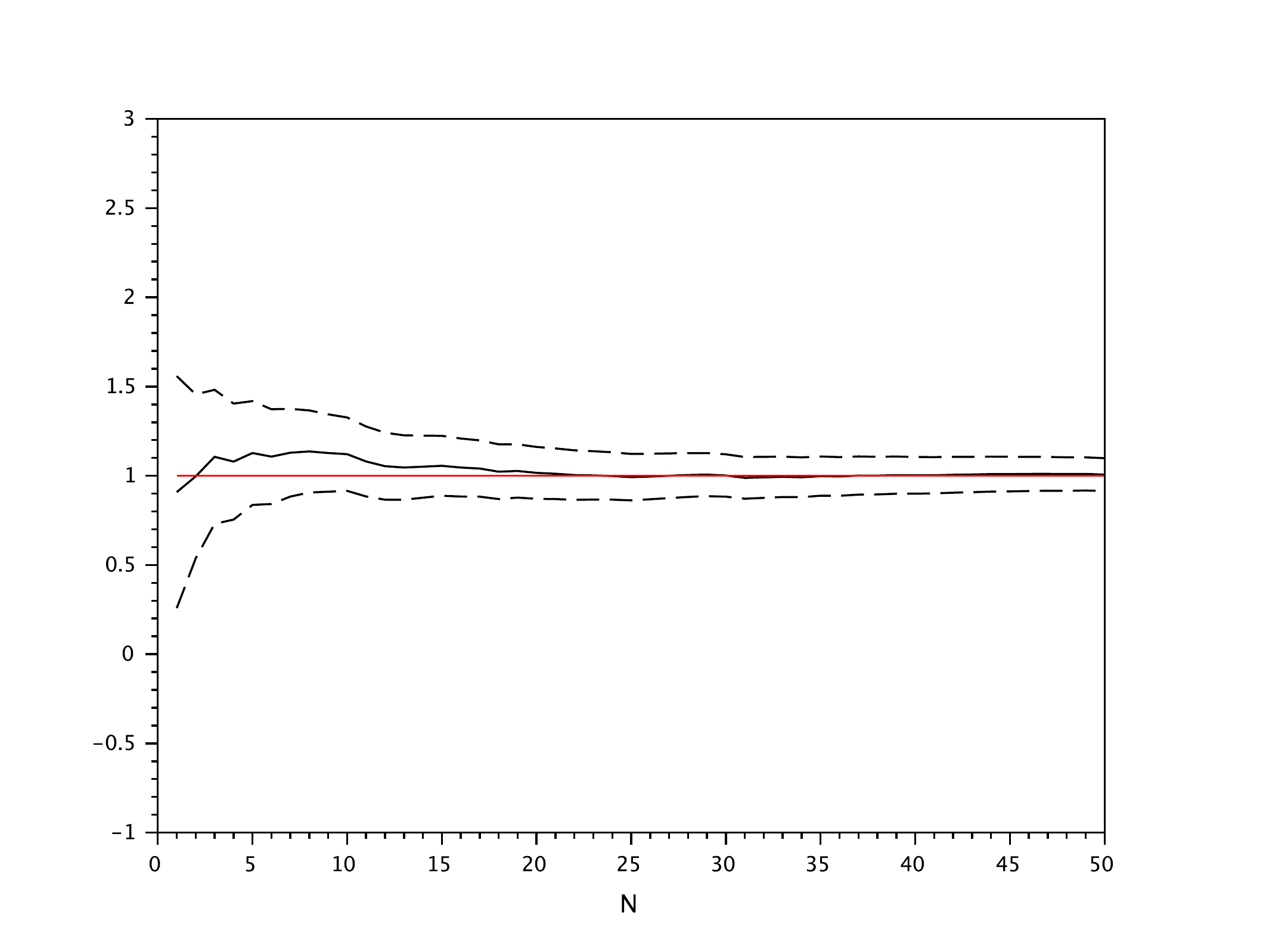}
\end{minipage}
\caption{Plots of $N\mapsto\widetilde\theta_N$ (black line) and of the bounds of the $95\%$-ACIs (dashed black lines) for Model 1 with $H = 0.7$ (left) and $H = 0.9$ (right).}
\label{plots_Model_1}
\end{figure}
\begin{figure}[H]
\begin{minipage}[b]{0.4\linewidth}
 \centering
 \includegraphics[width=7cm,height=6cm]{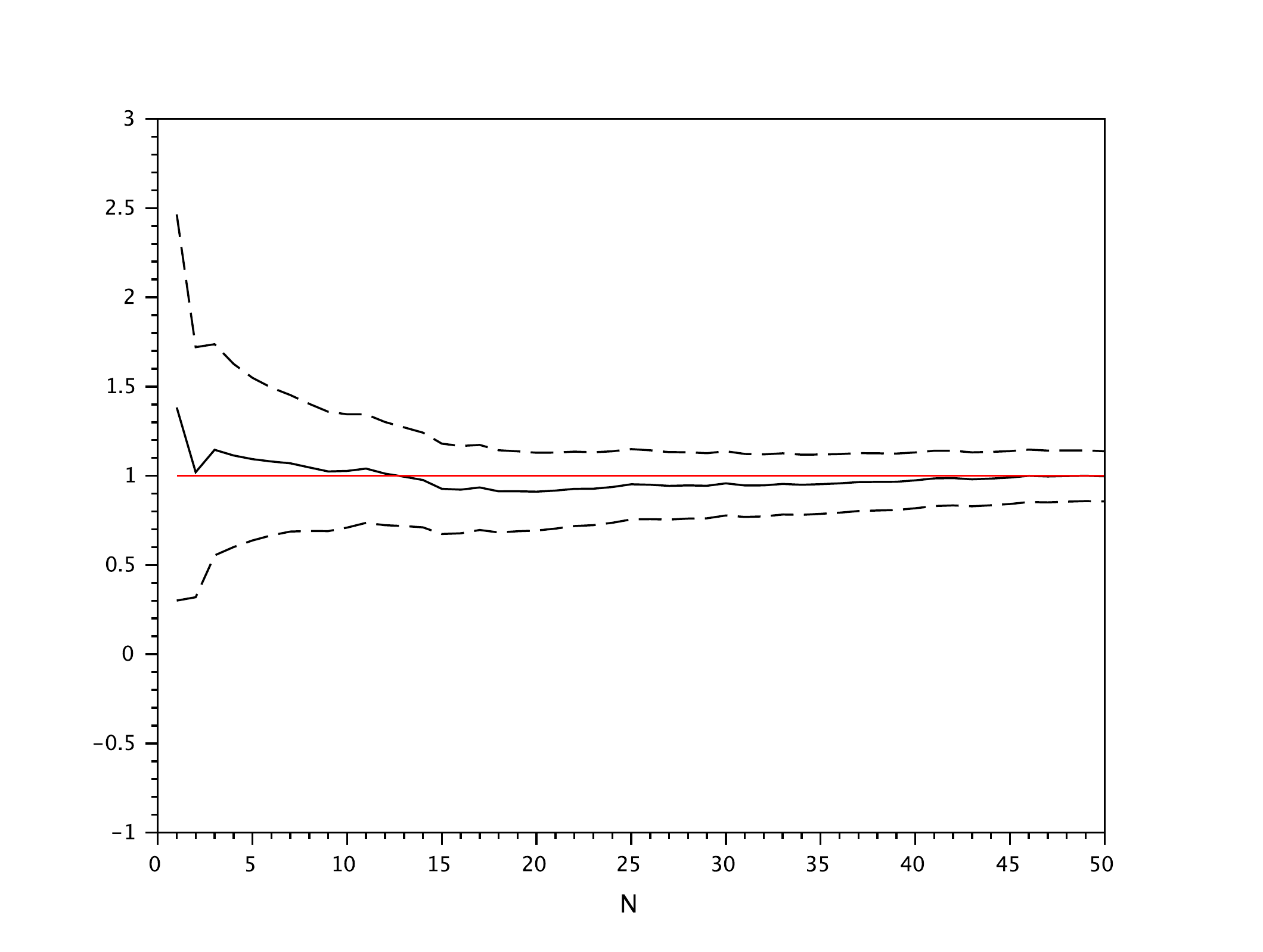}
\end{minipage}
\begin{minipage}[b]{0.4\linewidth}
 \centering
 \includegraphics[width=7cm,height=6cm]{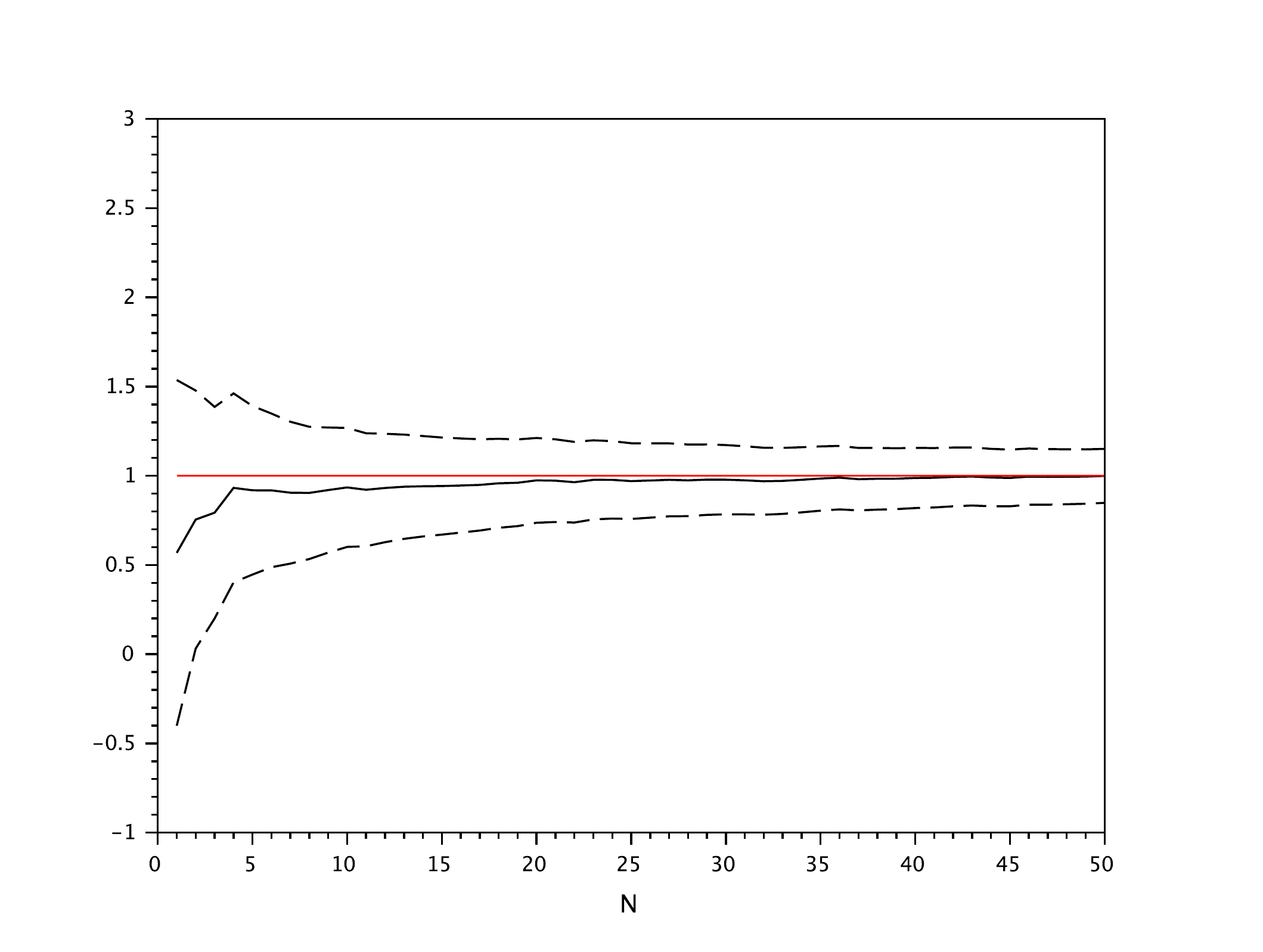}
\end{minipage}
\caption{Plots of $N\mapsto\widetilde\theta_N$ (black line) and the bounds of the $95\%$-ACIs (dashed black lines) for Model 2 with $H = 0.7$ (left) and $H = 0.9$ (right).}
\label{plots_Model_2}
\end{figure}
\noindent
Now, let us show that the way to choose the threshold $\mathfrak d$ for the (computable) truncated estimator of $\theta_0$ in Remarks \ref{remark_threshold_fBm_part_1} and \ref{remark_threshold_fBm_part_2}.(2) is also appropriate on the numerical side. Remark \ref{remark_threshold_fBm_part_1}.(1) applies to Model 1 with $\mathfrak d_{\max} =\pi^2/8\approx 1.23$, and Remark \ref{remark_threshold_fBm_part_1}.(2) applies to Model 2 with $\mathfrak d_{\max} = 25/2e^{-2\cdot 0.75}\approx 2.79$. With $H = 0.9$ and $N = 15$ (small), the mean error (for 100 datasets) of the $\mathfrak d$-truncated estimator for Model 1 (resp. Model 2) is plotted in Figure \ref{plots_thresholds} with respect to
\begin{displaymath}
\mathfrak d\in\{0.5 + 0.1k\textrm{ $;$ }k\in\{0,\dots,30\}\}
\quad\textrm{(resp. }
\mathfrak d\in\{1 + 0.5k\textrm{ $;$ }k\in\{0,\dots,30\}\}).
\end{displaymath}
For both Model 1 and Model 2, the threshold $\mathfrak d_{\max}$ recommended in Remark \ref{remark_threshold_fBm_part_2}.(2), which gives reasonable theoretical guarantees, is drawn in red, and one can observe that the mean error of the $\mathfrak d$-truncated estimator (black line) starts to explode with respect to that of $\widetilde\theta_{15}$ (black dashed line) for much higher values of $\mathfrak d$.
\begin{figure}[H]
\begin{minipage}[b]{0.4\linewidth}
 \centering
 \includegraphics[width=7cm,height=6cm]{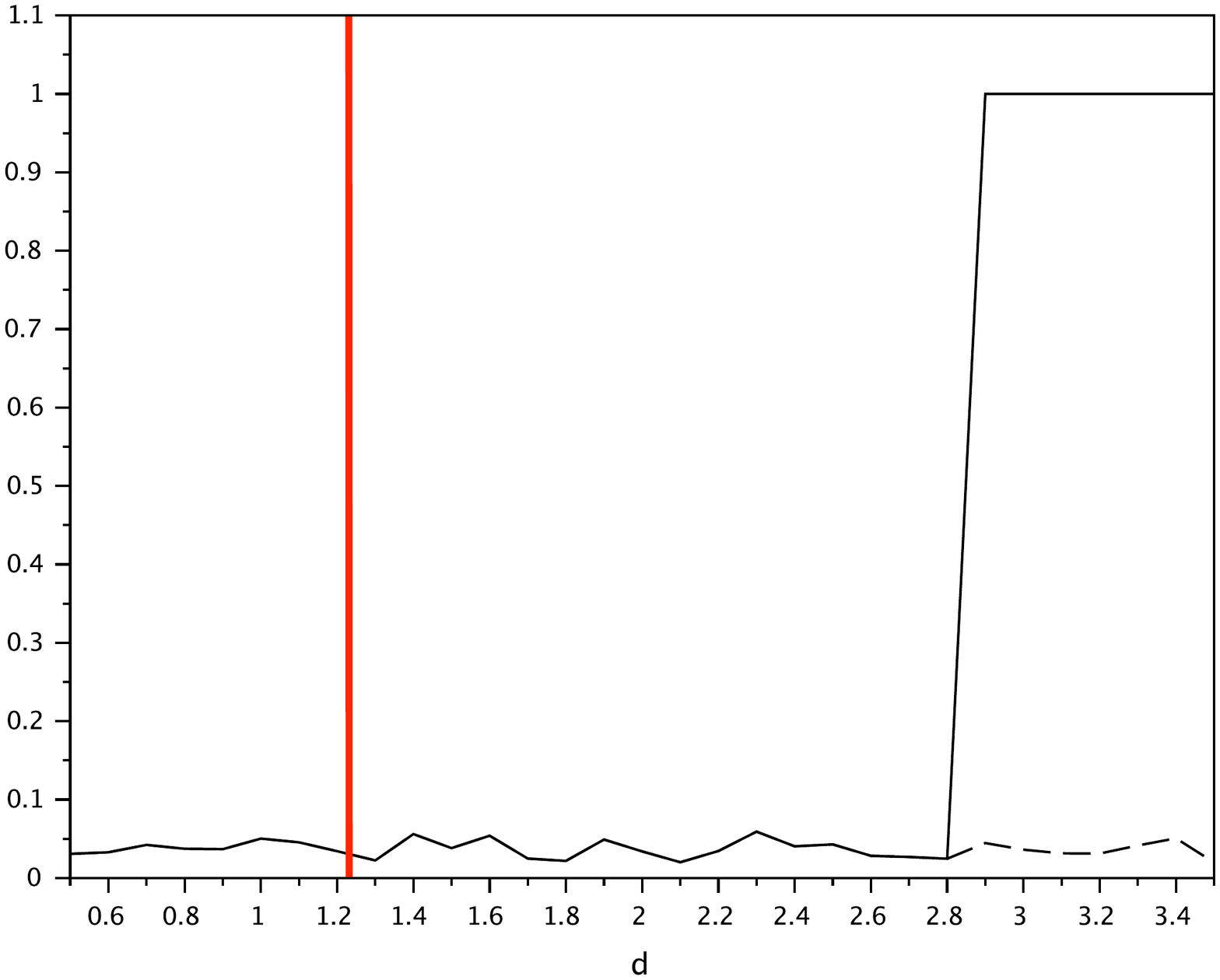}
\end{minipage}
\begin{minipage}[b]{0.4\linewidth}
 \centering
 \includegraphics[width=7cm,height=6cm]{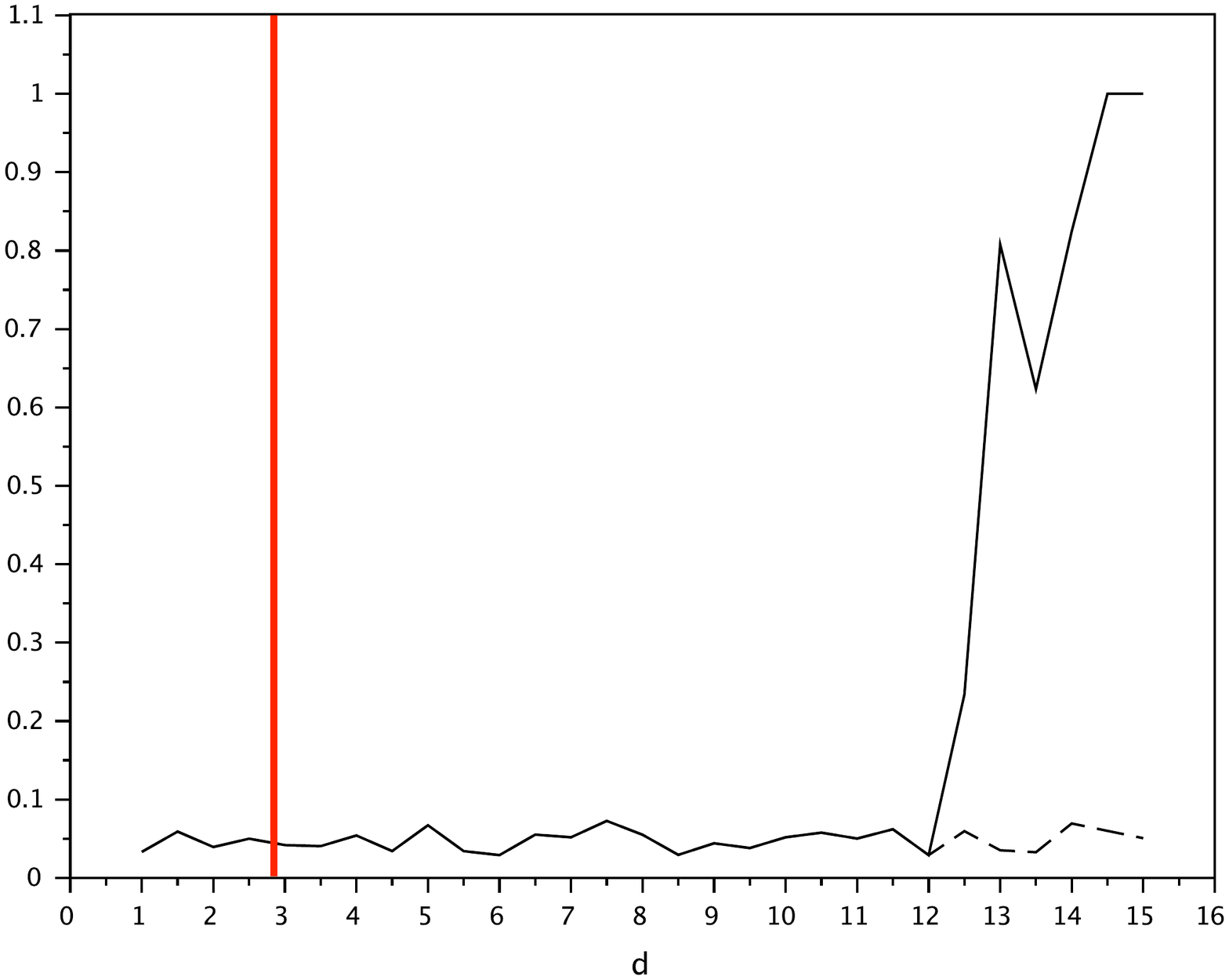}
\end{minipage}
\caption{Plots of the mean error of the $\mathfrak d$-truncated estimator (black line) for Model 1 (left) and Model 2 (right).}
\label{plots_thresholds}
\end{figure}
%


%
\section{Conclusion and perspectives}
The main contributions of our paper are:
\begin{enumerate}
 \item To provide a computable approximation $\widetilde\theta_N$ of $\widehat\theta_N$ when $H > 1/2$. Such approximation of $\widehat\theta_N$ is interesting in practice because it is well-known that the Skorokhod integral involved in the definition of $\widehat\theta_N$ is not computable when $H\neq 1/2$. For $H\in (1/3,1/2)$, the definition of $\widetilde\theta_N$ is extended in Section \ref{section_fBm_rough}.
 \item To provide an asymptotic confidence interval for $\widehat\theta_N$ (resp. $\widetilde\theta_{N}^{\mathfrak c = 1/2}$) when $H = 1/2$ (resp. $H > 1/2$) and $X^1,\dots,X^N$ are independent (see Proposition \ref{ACI_BM} (resp. Proposition \ref{ACI_fBm})).
 \item To establish risk bounds on $\widehat\theta_{N}^{\mathfrak d}$ and $\widetilde\theta_{N}^{\mathfrak c,\mathfrak d}$ when $X^1,\dots,X^N$ may be dependent (i.e. $\mathcal R_N\neq\emptyset$) (see Propositions \ref{risk_bound_BM}, \ref{risk_bound_fBm}, \ref{risk_bound_computable_estimator} and \ref{risk_bound_fBm_rough}).
 \item To establish a risk bound on a discrete-time approximation of $\widehat\theta_{N}^{\mathfrak d}$ when $H = 1/2$ and $X^1,\dots,X^N$ may be dependent (see Proposition \ref{risk_bound_discrete_time_BM}).
\end{enumerate}
Although it's out of the scope of the present paper which mainly focuses on continuous-time observations based estimators of $\theta_0$, one could study a discrete-time approximation of $\widetilde\theta_N$ as, for instance, $\widetilde\theta_{N,n} = I_N + R_{N,n}$ where $R_{N,n}$ is the fixed point (when it exists and is unique) of a discrete-time approximation of the map $\Phi_N$ introduced in Subsection \ref{section_computable_estimator}. Note also that our computable approximation method for $\widehat\theta_N$ could be extended to the (nonparametric) projection least squares estimator studied in Comte \& Marie \cite{CM21}. For a {\it regular enough} orthonormal family $(\varphi_1,\dots,\varphi_m)$ of $\mathbb L^2(\mathbb R,dx)$ with $m\in\{1,\dots,N\}$, \cite{CM21} deals with the estimator
\begin{displaymath}
\widehat b_m(.) :=\sum_{j = 1}^{m}\widehat\theta_j\varphi_j(.)
\quad {\rm with}\quad
\widehat\theta =
\widehat\Psi_{m}^{-1}\widehat Z_m
\end{displaymath}
of the whole function $b_0(.) =\theta_0b(.)$ in Equation (\ref{main_equation}), where
\begin{displaymath}
\widehat\Psi_m :=\left(\frac{1}{NT}\sum_{i = 1}^{N}\int_{0}^{T}\varphi_j(X_{s}^{i})\varphi_{\ell}(X_{s}^{i})ds\right)_{j,\ell}
\quad {\rm and}\quad
\widehat Z_m :=\left(\frac{1}{NT}\sum_{i = 1}^{N}\int_{0}^{T}\varphi_j(X_{s}^{i})\delta X_{s}^{i}\right)_j.
\end{displaymath}
By Proposition \ref{Skorokhod_Young_relationship} and by the chain rule for the Malliavin derivative, $\widehat Z_m =\gamma_N(b_0)$ with
\begin{displaymath}
\gamma_N(\psi) :=
\left(\frac{1}{NT}\sum_{i = 1}^{N}\left(
\int_{0}^{T}\varphi_j(X_{s}^{i})dX_{s}^{i} -
\alpha_H\sigma^2\int_{0}^{T}\int_{0}^{t}
\varphi_j'(X_{t}^{i})\exp\left[\int_{s}^{t}\psi'(X_{u}^{i})du\right]
|t - s|^{2H - 2}dsdt\right)\right)_j.
\end{displaymath}
Then,
\begin{displaymath}
\widehat b_m =\Gamma_N(b_0)\approx\Gamma_N(\widehat b_m)
\quad {\rm with}\quad
\Gamma_N(\psi)(.) :=\sum_{j = 1}^{m}[\widehat\Psi_{m}^{-1}\gamma_N(\psi)]_j\varphi_j(.).
\end{displaymath}
This legitimates to consider the estimator $\widetilde b_m(.) :=\sum_{j = 1}^{m}\widetilde\theta_j\varphi_j(.)$ of $b_0$ such that $\widetilde b_m =\Gamma_N(\widetilde b_m)$.
%


%

%
\end{document}